\documentclass{amsart}

\usepackage[utf8]{inputenc}
\usepackage[T1]{fontenc}
\usepackage[english]{babel}

\usepackage{fancyhdr}
\usepackage{graphicx}
\usepackage{float}
\usepackage{hyperref}
\usepackage{todonotes}
\usepackage{amsmath}
\usepackage{amssymb}
\usepackage{amsfonts}
\usepackage{amsthm}
\usepackage{mathrsfs}
\usepackage{verbatim}
\usepackage{breqn}
\usepackage{soul}
\usepackage{tikz-cd}
\usepackage{tikzit}
\usepackage{mathtools}
\usepackage{extpfeil}
\usepackage{caption}
\usepackage{subcaption}
\usepackage{extpfeil}
\usetikzlibrary{decorations.pathreplacing,angles,quotes}

\tikzstyle{filled}=[fill=black, draw=black, shape=circle, scale=0.3]

\tikzstyle{blue}=[-, draw=blue]

\newtheorem{mainthm}{Theorem}
\newtheorem{maincor}[mainthm]{Corollary}

\newtheorem{thm}{Theorem}
\numberwithin{thm}{subsection}
\numberwithin{equation}{section}

\newtheorem{lem}[thm]{Lemma}
\newtheorem{rmk}[thm]{Remark}

\newtheorem{cor}[thm]{Corollary}
\theoremstyle{definition}

\newcommand{\mult}{\mathrm{mult}}

\newcommand{\C}{\mathbb{C}}
\newcommand{\Ct}{\mathbb{C}{(\!(\!\hspace{0.7pt}t\hspace{0.7pt}\!)\!)}}

\newcommand{\D}{\mathbb{D}}
\newcommand{\Q}{\mathbb{Q}}

\newcommand{\X}{\mathscr{X}}
\renewcommand{\P}{\mathbb{P}}
\renewcommand{\O}{\mathcal{O}}
\newcommand{\an}{\mathrm{an}}

\newcommand{\R}{\mathbb{R}}

\newcommand{\hyb}{\mathrm{hyb}}
\newcommand{\N}{\mathbb{N}}
\newcommand{\zbar}{\overline{z}}

\renewcommand{\L}{\mathcal{L}}

\newcommand{\Sk}{\mathrm{Sk}}

\renewcommand{\div}{\mathrm{div}}
\newcommand{\red}{\mathrm{red}}

\newcommand{\Y}{\mathscr{Y}}

\renewcommand{\tilde}{\widetilde}

\newcommand{\val}{\mathrm{val}}
\newcommand{\F}{\mathcal{F}}
\newcommand{\res}{\mathrm{res}}

\newcommand{\CC}{\mathrm{CC}}
\newcommand{\ttheta}{\tilde{\theta}}
\newcommand{\tpsi}{\tilde{\psi}}
\newcommand{\B}{\overline{B}}
\newcommand{\wt}{\mathrm{wt}}

\newcommand{\Mgbar}{\overline{\mathcal{M}_g}}

\newcommand{\Cgbar}{\overline{\mathcal{C}_g}}

\newcommand{\Aut}{\mathrm{Aut}}

\title[Convergence of Narasimhan--Simha measures]{Convergence of Narasimhan--Simha measures on degenerating families of Riemann surfaces}
\author{Sanal Shivaprasad}
\date{\today}

\begin{document}
\maketitle

\begin{abstract}
Given a compact Riemann surface $Y$ and a positive integer $m$, \break Narasimhan and Simha defined a measure on $Y$ associated to the $m$-th tensor power of the canonical line bundle. 
  We study the limit of this measure on holomorphic families of Riemann surfaces with semistable reduction.
  The convergence takes place on a hybrid space whose central fiber is the associated metrized curve complex in the sense of Amini and Baker.
  We also study the limit of the measure induced by the Hermitian pairing defined by Narasimhan--Simha measure. 
  For $m = 1$, both these measures coincide with the Bergman measure on $Y$.
  We also extend the definition of the Narasimhan--Simha measure to the singular curves on the  boundary of $\Mgbar$ in such a way that these measures form a continuous family of measures on the universal curve over $\Mgbar$.
\end{abstract}

\section{Introduction}
Let $Y$ be a compact Riemann surface of genus $g \geq 1$ and $m$ a fixed positive integer. Let $\Omega_{Y}^{\otimes m}$ denote the $m$-th tensor power of the canonical line bundle on $Y$.
Using the global sections of $\Omega_{Y}^{\otimes m}$, Narasimhan and Simha defined a volume form $\tau^{(m)}$ on $Y$ as follows \cite{NS68}. Given $\theta \in H^0(Y,\Omega_{Y}^{\otimes m})$, let $|\theta|^{2/m}$ denote the associated volume form i.e.~if locally $\theta(z) = f(z)dz^{\otimes m}$, then $|\theta|^{2/m}(z) = |f(z)|^{2/m}(\frac{i}{2} dz \wedge d \overline{z})$. Then,
$$ \tau^{(m)}(z) := \max_{\{ \theta \mid \int_{Y} |\theta|^{2/m} = 1\}} |\theta|^{2/m}(z) $$
is a continuous positive volume form on $Y$ that we call the \emph{Narasimhan--Simha volume form} associated to the line bundle $\Omega_{Y}^{\otimes m}$. The induced Radon measure on $Y$ is called the \emph{Narasimhan--Simha measure} and is also denoted by $\tau^{(m)}$.

More generally, given points $P_1,\dots,P_r \in Y$ and integers $0 < b_1,\dots,b_r < m$, a similar construction yields the Narasimhan--Simha measure $\tau^{(m,b_1 P_1 +\dots+b_r P_r)}$ on $Y$ associated to the line bundle $\L = \Omega_Y^{\otimes m}(b_1P_1 + \dots + b_rP_r)$. For details, see Section \ref{subsecPluriBergman}.

When $m=1$, the Narasimhan--Simha measure associated to $\Omega_{Y}$, $\tau^{(1)}$, coincides with the \emph{Bergman measure} \cite[Section 4]{Ber10} on $Y$ (See Section  \ref{eqnGenBergmanPairing} for details). Thus, $\tau^{(m)}$ is a possible generalization of the Bergman measure using pluricanonical forms. The volume form $\tau^{(m)}$ was introduced by Narasimhan and Simha to construct the moduli space of projective complex structures on a given compact connected real analytic manifold \cite{NS68}. Tsuji, and Berndtsson--Păun studied the semipositivity of the curvature current of the Narasimhan--Simha metric in families \cite{Tsu07} \cite{BP08}.

The asymptotics of the Bergman measure in degenerating families is studied in \cite{HJ96}, \cite{Don15}, \cite{dJon19}, \cite{Shi20A} \cite{AN20}. We are interested in computing the asymptotics of the Narasimhan--Simha measure in degenerating families. 

Let $X \to \D^*$ be a holomorphic family of curves of genus $g$ and $m \geq 2$ be a fixed integer. Let $B = b_1 B_1 + \dots + b_rB_r$ be a horizontal divisor on $X$ for integers $0 < b_1,\dots,b_r < m$. Denote $\L = \Omega_{X/\D^*}^{\otimes m}(B)$. Assume that if $g = 1$, then $r \geq 1$ and if $g = 0$, $r \geq 3$ and $\deg(B|_{X_t}) \geq 2m$. Note that when $g=0$, the assumption $r \geq 3$ and $\deg(B|_{X_t}) \geq 2m$ is equivalent to requiring that $H^0(X_t,\L_{X_t}) \neq 0$. Therefore, effectively, the only case we are excluding is when $g = 1$ and $r = 0$. But in this case, Theorem \ref{mainThmA}, stated below, is already known \cite[Theorem B]{Shi20A} (see also \cite[Corollary 4.8]{CLT10}, \cite[Theorem C]{BJ17} and \cite[Remark 16.4]{dJon19} for related results).

We will also assume that $(X,B_\red)$ has semistable reduction i.e.~there exists a regular model $\X$ of $X$ such that $\X_0$ is reduced and $\X_0 + \overline{B}_\red$ is an snc divisor on $\X$, where $\overline{B}$ denotes the component-wise closure of $B$ in $\X$. A theorem of Deligne and Mumford  guarantees that such a model always exists after a base change $\D^* \to \D^*$ given by $t \mapsto t^N$ for some positive integer $N$ \cite{DM69}. 

Let $\tau^{(m,B)}_t := \tau^{(m,B|_{t})}$ be the Narasimhan--Simha measure on $X_t$ associated to $\L|_{X_t}$.
If $\X$ is an snc model of $X$, then it is natural to ask what the limit of $\tau^{(m,B)}_t$ on $\X$ is. However, instead of computing the limit of $\tau^{(m,B)}_t$ on $\X$, we instead compute the limit of $\tau^{(m,B)}_t$ on the \emph{metrized curve complex hybrid space},  $\X_\CC^\hyb$, defined in \cite{Shi20A}. This has the advantage of simultaneously computing the limit of $\tau^{(m,B)}_t$ in $\X$ as well as in the non-Archimedean hybrid space in the sense of Boucksom and Jonsson \cite{Ber09} \cite{BJ17}.

The space $\X_\CC^\hyb$ is a partial  compactification of $X$ with central fiber being the \emph{metrized curve complex} $\Delta_\CC(\X)$ (in the sense of Amini and Baker \cite{AB15}) associated to $\X_0$. The latter is obtained by replacing all nodal points in $\X_0$ with  line segments i.e.~by taking the normalization $\tilde{\X_0}$, of $\X_0$ and adding a line segment connecting $P',P'' \in \tilde{\X_0}$ if $P',P''$ lie over the same nodal point in $\X_0$. See Figure \ref{figExampleGenus4Family} for an example. We refer to the irreducible components of $\tilde{\X_0}$ as \emph{curves} in $\Delta_\CC(\X)$ and the line segments as \emph{edges} in $\Delta_\CC(\X)$.

\begin{figure}
\begin{subfigure}{0.42\linewidth}
  \includegraphics[scale=0.3]{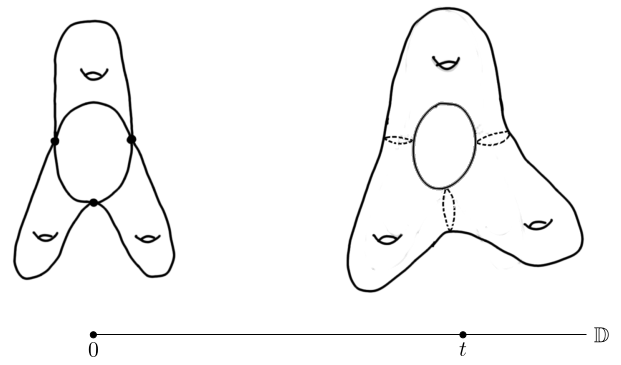}
  \caption{A degenerating family of genus 4 curves}
\end{subfigure}
\hspace{0.1\linewidth}
\begin{subfigure}{0.42\linewidth}
  \includegraphics[scale=0.3]{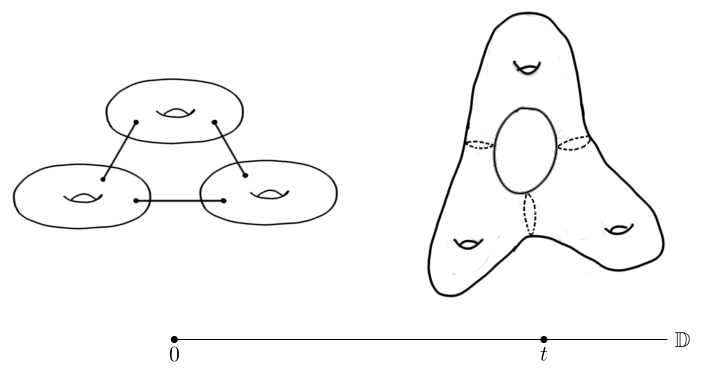}
  \caption{The associated metrized curve complex hybrid space}
\end{subfigure}
\caption{A family of genus 4 curves and the associated curve complex hybrid space}
\label{figExampleGenus4Family}  
\end{figure}

We have the following theorem regarding the convergence of $\tau^{(m,B)}_t$. 
\begin{mainthm}
\label{mainThmA}
There exists a measure $\tau^{(m,B)}_0$ on $\Delta_\CC(\X)$ such that $\tau^{(m,B)}_t \to \tau^{(m,B)}_0$ weakly as measures on $\X^\hyb_\CC$. 
\end{mainthm}

The limiting measure $\tau^{(m,B)}_0$ is a sum of Narasimhan--Simha measures on the curves  and Lebesgue measures on the edges. For details, see Theorem \ref{thmA}. The reason for working with the more general Narasimhan-Simha measure $\tau^{(m,B)}$ instead of just $\tau^{(m)}$ is that even if we start with $\tau^{(m)}_t$, the restriction of the limiting measure to a curve $E$ in $\Delta_\CC(\X)$ could still be of the form $\tau^{(m,a_1P_1+\dots+a_sP_s)}$. The mass of an edge $e$ under $\tau_0^{(m,B)}$ is $\frac{1}{N}$, where $N$ is the length of the maximal inessential chain containing $e$ (see Section \ref{subsecDualGraph} for details).

As a corollary of Theorem \ref{mainThmA}, we get that the limit of $\tau^{(m,B)}_t$ on $\X$ is a sum of Narasimhan--Simha measures on certain irreducible components of $\X_0$ and Dirac masses on nodal points of $\X_0$. The Dirac mass at a nodal point is equal to the mass of the associated edge in $\Delta_\CC(\X)$ with respect to the limiting measure $\tau^{(m,B)}_0$.

The Berkovich hybrid space $X^\hyb = X \cup X^\an_\Ct$ is a partial compactification of $X$ with  central fiber being the Berkovich analytification of $X$ viewed as a variety over $\Ct$. This partial compactification has the advantage that it does not depend on the choice of a model $\X$ and is, therefore, canonical. A number of degeneration problems \cite{BJ17}, \cite{Oda17}, \cite{Sus18}, \cite{LS19}, \cite{Sch19}, \cite{Shi19}, \cite{Shi20A} \cite{Li20} and dynamical problems \cite{Fav18}, \cite{DF19}, \cite{DKY19} have been studied in this setting. Here, we compute the limit of $\tau^{(m,B)}_t$ on $X^\hyb$ as a corollary of Theorem \ref{mainThmA}.

\begin{maincor}
\label{mainCorA}
The measures $\tau^{(m,B)}_t$ admit a weak limit as $t \to 0$ on $X^\hyb$. The support of this limiting measure coincides with the \emph{essential skeleton} of the pair $(X_\Ct,\frac{1}{m}B_\Ct)$  and is given by a sum of Lebesgue measures on edges and Dirac masses on vertices. 
\end{maincor}

The \emph{essential skeleton}  of the pair $(X_\Ct,\frac{1}{m}B_\Ct)$ (see Section \ref{subsecEssentialSkeleton} for details) is a piecewise linear subset of $X^\hyb$ which encodes information about the pair $(X_\Ct,\frac{1}{m}B_\Ct)$ \cite{KS06} \cite{MN15} \cite{BM19}.

Moreover, in the case when $B = 0$, the measure $\tau_0^{(m)}$ is independent of the one parameter family $X$. More precisely, if $\X \to \D$ and $\Y \to \D$ are two families of genus $g$ curves degenerating to the same semistable curve $C = \X_0 = \Y_0$, then the limit of the Narasimhan-Simha measure with respect to $\Omega_{X_t}^{\otimes m}$ and with respect to $\Omega_{Y_t}^{\otimes m}$ coincide on $\Delta_\CC(C)$. 
Since the data of the one-parameter family keeps track of the `direction of approach' towards a stable curve in $\Mgbar$, we could ask whether the Narasimhan-Simha measures with respect to $\Omega^{\otimes m}$ form a continuous family of measures on $\Mgbar$.

To make this precise, we first extend the notion of Narasimhan-Simha measure to all stable curves by considering the limiting measure described in Theorem \ref{mainThmA} and collapse all edges to a node. Note that this measure will place a unit Dirac mass at all nodal points of a stable curve.

Now consider the universal curve $\Cgbar \to \Mgbar$. Recall that, topologically, the fiber of this map over the isomorphism class of a stable curve $C$ is $C/\Aut(C)$. 
 Let ${\tau'}^{m}_{C}$ denote the pushforward of $\tau^{(m)}$ from $C$ to $C/\Aut(C)$. Let $(\mathcal{C}^0(\Cgbar))^{\vee}$ denote the space of Radon measures on $\Cgbar$ equipped with the weak$^*$ topology. 
\begin{mainthm}
  \label{mainThmC}
  The map $\Mgbar \to (\mathcal{C}^0(\Cgbar))^{\vee}$ given by $[C] \mapsto {\tau'}^{m}_{C}$ is continuous.    
\end{mainthm}

Note that the above result is in a stark contrast with the case for the Bergman measures (i.e.~when $m = 1$), where the mass of limiting measure on an edge depends strongly on the lengths of all the edges in the dual graph. Thus, an analog of Theorem \ref{mainThmC} would be false in the case of Bergman measures2. In fact, to extend the Bergman measures continuously, Amini and Nicolussi construct a large hybrid space which keeps track of the relative orders of the logarithmic rates of approach to each node on a stable curve \cite{AN20}.

We also compute the asymptotics of a measure closely related to the Narasimhan-Simha measure. Let $Y$ be a compact Riemann surface of genus $g$, $P_1,\dots,P_r$ points on $Y$ and $0 < b_1,\dots,b_r < m$ integers. 
We get a Hermitian pairing on $H^0(Y,\Omega_{Y}^{\otimes m}(b_1P_1+\dots+b_r P_r))$ given by
$$ \langle \theta,\vartheta \rangle = \left(\frac{i}{2}\right)^{m} \int_{Y} \frac{\theta \wedge \overline{\vartheta}}{(\tau^{(m,b_1P_1 + \dots + b_rP_r)})^{m-1}}.$$
Let $e_1,\dots,e_M$ be an orthonormal basis of $H^0(Y,\Omega_Y^{\otimes m}(b_1P_1+\dots+b_r P_r))$ with respect to the above pairing. Then, the positive volume form $$\mu^{(m,b_1P_1+\dots+b_rP_r)} = \left(\frac{i}{2}\right)^{m}\sum_{i=1}^M \frac{e_i \wedge \overline{e_i}}{(\tau^{(m,b_1P_1 + \dots + b_rP_r)})^{m-1}}$$ does not depend on the choice of the orthonormal basis and we call it as the \emph{pluri-Bergman measure} on $Y$ associated to $\Omega_Y^{\otimes m}(b_1P_1+\dots+b_r P_r)$. Note that when $m = 1$ and $r=0$, $\mu^{(1)}$ is just the Bergman measure. Thus, this measure is yet another generalization of the Bergman measure. 

Consider the family $X \to \D^*$ along with the horizontal divisor $B$. Using the same notation as before, let $\mu^{(m,B)}_t$ denote the measure $\mu^{(m,B_t)}$ on $X_t$ associated to $\L_t$. We are also able to compute the limit of $\mu^{(m,B)}_t$ on $\X^\hyb_\CC$.

 \begin{mainthm}
  \label{mainThmB}
There exists a measure $\mu^{(m,B)}_0$ on $\Delta_\CC(\X)$ such that $\mu^{(m,B)}_t \to \mu^{(m,B)}_0$ weakly as measures on $\X^\hyb_\CC$. 
\end{mainthm}
The measure $\mu^{(m,B)}_0$ is also a sum of pluri-Bergman measures on the curves in $\Delta_\CC(\X)$ and Lebesgue measures on the edges. The mass of each edge with respect to $\mu^{(m,B)}_0$ is the same as the mass of the edge with respect to $\tau^{(m,B)}_0$. For details, see Section \ref{secConvergenceCanonical} and Theorem \ref{thmB}.

As before, we also get the limit of measures $\mu^{(m,B)}_t$ on the hybrid space $X^\hyb$. 
\begin{maincor}
  \label{mainCorB}
  The measures $\mu^{(m,B)}_t$ converge to a measure on the hybrid space $X^\hyb$ whose support is the essential skeleton of the pair  {$(X_\Ct,\frac{1}{m}B_\Ct)$}. The limiting measure is a sum of Dirac masses on the vertices and Lebesgue measure on the edges. 
\end{maincor}

Finally, we would also like to understand what happens to the limiting pluri-Bergman measure as $m \to \infty$. We compute this limit on $X^\hyb$ instead of $\X^\hyb_\CC$. The reason for doing so is that it is not clear to us what the limit of $\mu^{(m)}$ as $m \to \infty$ is for a fixed Riemann surface. However, the total mass of $\mu^{(m)}$ is easy to figure out. 

There are two ways to think of the limit. In the first case, we fix $B$ and suppose that $g \geq 2$.  Let $\mu_{t}^{(m,B)}$ denote the pluri-Bergman measure on $X_t$ induced by $\Omega_{X_t}^{\otimes m}(B|_{X_t})$. By abuse of notation, let $\mu_{0}^{(m,B)}$ also denote the weak limit of $\mu_t^{(m,B)}$  on $X^\hyb$ as $t \to 0$. Then, the measures $\mu_{0}^{(m,B)}$, normalized to volume $2g-2$, converges to an analogue of the hyperbolic measure on $X^\an_\Ct$  and this limit does not depend on the choice of the divisor $B$ (see Section \ref{subsecLimitMToInfty}). This limit measure lives on the dual graph of the stable reduction of $X$ (which coincides with the essential skeleton in this case). It places no mass on the edges and places a mass of $2g(v) - 2 + \val(v)$ on each vertex, where $g(v)$ is the genus of the irreducible component associated to $v$ and $\val(v)$ is the valency of the vertex $v$ in the dual graph. It follows from \cite{Sch19} that this measure is the limit of hyperbolic measures on $X_t$. It seems to be unknown whether the measures $\mu_t^{(m)}$, normalized to volume $2g-2$, themselves converge weakly to the hyperbolic measure on $X_t$. 

Another way to think of the limit is to fix the $\Q$-divisor $\frac{1}{m}B$ instead of fixing $B$ i.e. we consider $\mu^{(km,kB)}_0$ associated to $\omega_{X_t}^{\otimes km}(kB|_{X_t})$ and take the limit as $k \to \infty$. Assume that $2g - 2 + \frac{\deg(B|_{X_t})}{m} > 0$. 
In this case, there exists a model $\X$ such that $\mu^{(km,kB)}_0$, normalized to volume $2g-2 + \frac{\deg(B|_{X_t})}{m}$, converges to a measure that places no mass on the edges and places a Dirac mass of $2g(v_E) - 2 + \val(v_{E}) + \frac{\deg(\B|_E)}{m}$ on each vertex $v_{E}$ of the dual graph of $\X_0$.

As for the limit of $\tau^{(m,B)}_0$, it is not even clear to us what the asymptotics of $\int_{X_t} \tau_t^{m,B}$ is as $m \to \infty$. In the case of $B = 0$, Tsuji has shown that the supremum of $\frac{\tau^{(m)}}{\int \tau^{(m)}}$ as $m \to \infty$ exists as a bounded volume form  \cite[Theorem 4.1]{Tsu07}. However, it is not clear whether this  supremum is a limit or not.  

\subsection*{Method of proof} A general observation (Lemma \ref{lemLimitOnOnlyOneModel}) tells that us in order to compute the limit on all snc models of $X$, it is enough to compute it on any one snc model. We work with the model $\X$ of $X$ obtained from the minimal snc model of $(X,B)$ by repeatedly blowing down the $(-1)$-curves $E$ in the central fiber for which $\deg(\B|_{E}) < m$. The advantage of working with this model is that $h^0(\omega_{\X_0}(\B|_{\X_0})) = h^0(\omega_{X_t}(B|_{X_t}))$; now we can apply Grauert's lemma \cite[Corollary III.12.9]{Har77} to find sections of $\omega_{\X/\D}(\B)$ that restrict to a basis of $H^0(\omega_{\X_0}(\B|_{\X_0}))$ and $H^0(\omega_{X_t}(B|_{X_t}))$. We make a clever choice of such a basis that renders the computations simple. By analyzing these sections, we find expressions for $\tau^{(m)}_t$ and $\mu^{(m)}_t$. Now, understanding the asymptotics of these sections allows us to understand the asymptotics of $\tau^{(m)}_t$ and $\mu^{(m)}_t$.

Theorems \ref{mainThmC} and \ref{mainThmB} also follow from similar calculations.

We prove Corollary \ref{corConvergenceBerkovichHybridSpace}, a general result on how to transfer convergence from $\X^\hyb_\CC$ to $X^\hyb$. As a consequence, Corollaries \ref{mainCorA} and \ref{mainCorB} follow from Theorems \ref{mainThmA} and \ref{mainThmB}, respectively, using Corollary \ref{corConvergenceBerkovichHybridSpace}.

\subsection*{Further questions}
On a fixed Riemann surface, we can consider a sequence of measures constructed iteratively using the recipe for constructing the pluri-Bergman measure, starting with the Bergman measure. Tsuji has shown that this sequence of measures converges to the hyperbolic measure \cite{Tsu10}. It would be interesting to know what the limit of these measures on $\Delta_\CC(\X)$ is and whether the sequence of limiting measures could be given a dynamical interpretation. 


We could also ask whether the measures $\mu^{(m)}$ converge to the hyperbolic measure as $m \to \infty$, which is the case for their limits on $X^\hyb$.

It would be interesting to know if there is a higher dimensional analog of Theorem \ref{mainThmA} and Corollary \ref{mainCorA}. 

\subsection*{Organization of the paper}
We discuss some preliminaries in Section \ref{secPreliminaries}. In Section \ref{secMCCHS}, we discuss the metrized curve complex hybrid space.
In Section \ref{secSectionsOfOmegaMB}, we discuss the global sections of $\omega_{\X_0}^{\otimes m}(\B|_{\X_0})$. We prove Theorem \ref{mainThmA} in Section \ref{secConvergenceNS} and in Section \ref{secConvergenceCanonical}, we prove Theorem \ref{mainThmB}. In Section 7, we prove Theorem \ref{mainThmC} 

\subsection*{Acknowledgments}
I would like to thank my advisor, Mattias Jonsson, for his suggestions and comments. This work was supported by the NSF grants DMS-1600011 and DMS-1900025.

\section{Preliminaries}
\label{secPreliminaries}
\subsection{Families of curves and models}
\label{subsecModels}
Let $\D$ denote the complex unit disk and let $\D^*$ denote the complex unit disk punctured at the origin. A family of curves $X \to \D^*$ of genus $g$ is a complex manifold $X$ along with a projective holomorphic submersion $X \to \D^*$ such that the fibers are smooth compact connected complex curves of genus $g$. We also assume that our family of curves is meromorphic at the origin i.e. $X \subset \P^N \times \D^*$  is cut out by polynomials whose coefficients are holomorphic functions on $\D^*$ and meromorphic at the origin. 

A  \emph{model} of  $X$ is a normal complex analytic space $\X$ along with a projective flat holomorphic map $\X \to \D$ such that $\X|_{\D^*} = X$ as complex analytic spaces over $\D^*$. Let $\X_0$ denote the central fiber of $\X$. Note that $\X_0$ will always be connected \cite[Corollary 8.3.6]{Liu02}. 
A model $\X$ is said to be an \emph{snc model} of $X$ if $\X$ is regular and $\X_{0,\red}$ is an snc divisor on $\X$. 

Given two snc models $\X'$ and $\X$ of $X$, we say that $\X'$ dominates $\X$ if there is a proper map $q\colon \X' \to \X$ such that $q|_{X}$ is the identity map. Note that $q$ is a bimeromorphic map between $\X'$ and $\X$.  

Let $B = b_1B_1 + \dots + b_rB_r$ be a horizontal divisor on $X$. After shrinking the base disk, we may assume that $B_i \cap B_j = \emptyset$ for $i \neq j$. Let $\overline{B}$ denote the component-wise closure of $B$ in $\X$. We say that $\X$ is an snc model of $(X,B)$ if $\X$ is regular and $(\X_0 + \B)_{\red}$ is an snc divisor on $\X$.

Let $m \geq 2$ be a positive integer. Suppose that $b_i < m$ for all $i$. Further assume that $\deg(B|_{X_t}) \geq 1$ if $g=1$ and $\deg(B|_{X_t}) \geq 2m$ if $g=0$. Throughout this paper, we will only be working with such pairs. Note that in this case $(X,B_\red)$ is stable in the sense of \cite{DM69} i.e.~$2g-2 + \deg(B_\red) > 0$. 

A theorem of Deligne and Mumford guarantees that if $(X,B_\red)$ is a stable pair, then after a base change $\D^* \to \D^*$ given by $u \mapsto u^N$, there exists an snc model $\X'$ of $(X,B)$ such that $\X'_0$ is reduced. Such an $\X'$ is called a \emph{semistable model} and we will assume that all our families have a semistable model. If $(X,B)$ has a semistable model, there exists a unique minimal one. Here, minimality means that any other semistable model is obtained by applying a sequence of blowups to the minimal semistable model. We can get the minimal semistable model by considering the stable reduction of $(X,B_\red)$ and then resolving the singularities by blowing up (see \cite[Chapter X.4]{ACG11}). 

Let $\X'$ denote the minimal semistable model of $(X,B)$. We will mostly work with the model $\X$ that is obtained from $\X'$ by repeatedly contracting the $(-1)$-curves $E$ in the central fiber for which $\deg(\B|_{E}) < m$. We will call $\X$ as the \emph{minimal snc model} of $(X,\frac{1}{m}B)$. The choice of notation is due to the fact that the model only depends on the $\Q$-divisor $\frac{1}{m}B$, in the sense that the minimal snc model of $(X,\frac{1}{m}B)$ is the same as that of $(X,\frac{1}{km}kB)$ for any positive integer $k$. Note that $\X_0 + \B_\red$ is no longer an snc divisor; however this is not a major problem as $\X_0$ is still a (reduced) snc divisor. Moreover, $\B$ does not pass through any nodal points of $\X_0$ -- this is due to the fact that the image of a $(-1)$-curve $E$ in $\X'$ under the map $\X' \to \X''$, obtained by contracting $E$, is a smooth point in $\X_0''$. For an example, see Figure \ref{figMinSNCModel}.

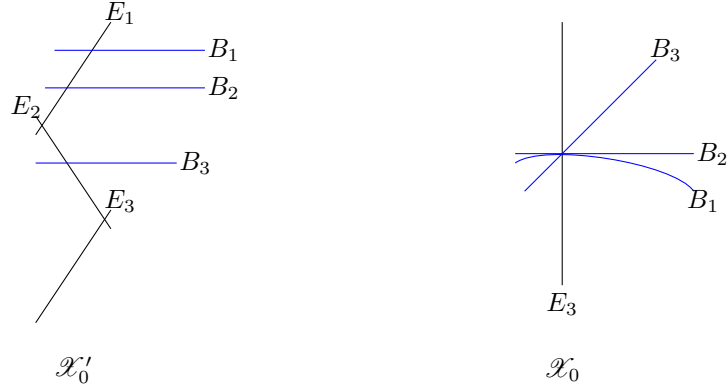
\begin{figure}
\begin{tikzpicture}
	\begin{pgfonlayer}{nodelayer}
		\node [style=none] (0) at (1, 2) {};
		\node [style=none] (1) at (-1, -1) {};
		\node [style=none] (2) at (-1, -0.5) {};
		\node [style=none] (3) at (1, -3.5) {};
		\node [style=none] (4) at (1, -3) {};
		\node [style=none] (5) at (-1, -6) {};
		\node [style=none] (6) at (-0.5, 1.25) {};
		\node [style=none] (7) at (3.5, 1.25) {};
		\node [style=none] (8) at (-0.75, 0.25) {};
		\node [style=none] (9) at (3.5, 0.25) {};
		\node [style=none] (10) at (-1, -1.75) {};
		\node [style=none] (11) at (2.75, -1.75) {};
		\node [style=none] (12) at (1.25, 2.25) {$E_1$};
		\node [style=none] (13) at (-1.25, -0.25) {$E_2$};
		\node [style=none] (14) at (1.25, -2.75) {$E_3$};
		\node [style=none] (15) at (4, 1.25) {$B_1$};
		\node [style=none] (16) at (4, 0.25) {$B_2$};
		\node [style=none] (17) at (3.25, -1.75) {$B_3$};
		\node [style=none] (18) at (13, 2) {};
		\node [style=none] (19) at (13, -5) {};
		\node [style=none] (20) at (11.75, -1.5) {};
		\node [style=none] (21) at (13, -1.5) {};
		\node [style=none] (22) at (16.5, -1.5) {};
		\node [style=none] (23) at (12, -2.5) {};
		\node [style=none] (24) at (15.5, 1) {};
		\node [style=none] (25) at (11.75, -1.75) {};
		\node [style=none] (26) at (16.5, -2.5) {};
		\node [style=none] (27) at (15.75, 1.25) {$B_3$};
		\node [style=none] (28) at (17, -1.5) {$B_2$};
		\node [style=none] (29) at (16.75, -2.75) {$B_1$};
		\node [style=none] (30) at (13, -5.5) {$E_3$};
		\node [style=none] (31) at (0.75, -7.25) {};
		\node [style=none] (32) at (0, -7.25) {$\X'_0$};
		\node [style=none] (33) at (13, -7.25) {$\X_0$};
	\end{pgfonlayer}
	\begin{pgfonlayer}{edgelayer}
		\draw (0.center) to (1.center);
		\draw (2.center) to (3.center);
		\draw (4.center) to (5.center);
		\draw [style=blue] (6.center) to (7.center);
		\draw [style=blue] (8.center) to (9.center);
		\draw [style=blue] (10.center) to (11.center);
		\draw (18.center) to (19.center);
		\draw [style=blue] (20.center) to (22.center);
		\draw [style=blue] (23.center) to (24.center);
		\draw [style=blue, bend left=45, looseness=0.50] (25.center) to (26.center);
	\end{pgfonlayer}
\end{tikzpicture}
  \caption{The above figure shows the minimal snc model, $\X$, for $(X,\frac{1}{4}B)$, where $B = B_1 + B_2 + B_3$ and $\X_0' = E_1 + E_2 + E_3$. The model $\X$ is obtained by first contracting $E_1$ and then $E_2$ from the minimal semistable model, $\X'$ , of $(X,B)$.}
\label{figMinSNCModel}  
\end{figure}

Let $\X$ be the minimal snc model of $(X,\frac{1}{m}B)$. For any irreducible component $E \subset \X_0$, let $\val(E)$ denote the number of intersection points of $E$ with the rest of $\X_0$ i.e.~$\val(E) = \#(E \cap \overline{(\X_0 \setminus E)})$. Any irreducible component $E \subset \X_0$ is of one of the following forms:
\begin{itemize}
\item $g(E) \geq 2$;
\item $g(E) = 1$ and either $\val(E) \geq 1$ or $\deg(\B|_E) \geq 1$;
\item $g(E) = 0$ and $\val(E) \geq 2$;
\item $g(E) = 0$, $\val(E) = 1$ and $\deg(\B|_E) \geq m$; or
\item $g(E) = 0$, $\val(E) = 0$ and $\deg(B|_{E}) \geq 2m$.
\end{itemize}


Note that all irreducible components $E \subset \X_0$ satisfy $2g(E)-2+\val(E)+\frac{\deg(\B|_E)}{m} \geq 0$. If $g(E) = 0, \val(E) = 2$ and $\deg(\overline{B}|_{E}) = 0$, we call $E$ as \emph{inessential}, otherwise we call it as \emph{essential}. If $B=0$, the essential irreducible components of $\X_0$ are exactly those that show up in the central fiber of the stable reduction of $X$ in the sense of \cite{DM69} i.e.~when $B = 0$, the inessential components correspond to $(-2)$-curves which can be contracted to obtain the stable reduction of $X$.

\subsection{Dual graph and the stable dual graph of a model}
\label{subsecDualGraph}
Let $(X,B)$ be a pair as from the previous section. 
Given an snc model $\X'$ of $X$, the \emph{dual graph} $\Gamma_{\X'}$ is a graph whose vertices correspond to irreducible components on $\X'_0$ and edges correspond to the nodes in $\X'_0$. Note that $\Gamma_{\X'}$ is allowed to have multiple edges between a pair of vertices, but no loops are allowed. Note that $\Gamma_{\X'}$ is connected because $\X_0'$ is connected. Associated to each vertex $v_E \in V(\Gamma_{\X'})$, we keep track of two numbers: the genus $g(v_E) = g(E)$ and the valency $\val(v_E) = \val(E)$.

We define the length of an edge $e$ between $v_{E_1}$ and $v_{E_2}$ by
$$l_e = \frac{1}{\mult_{\X'_0}(E_1) \cdot \mult_{\X'_0}(E_2)}.$$
In particular, when $\X'_0$ is reduced, all edges in $\Gamma_{\X'}$ have length 1. 

Now let $\X$ be the minimal snc model of $(X,\frac{1}{m}B)$. We call a vertex $v_E \in V(\Gamma_\X)$ inessential (respectively essential) if $E$ is inessential (respectively essential).

 Let $v_0, v_N \in V(\Gamma_\X)$ for some $N \geq 1$.  We define an \emph{inessential chain} between $v_0, v_N$ to be a sequence of vertices $v_0,\dots,v_N \in V(\Gamma_{\X})$ such that there is an edge between $v_{i-1}$ and $v_i$ for $1 \leq i \leq n$ and $v_1,\dots,v_{N-1}$ are inessential. Such a chain is said to be \emph{maximal} if $v_0$ and $v_N$ are essential. Note that we do allow $v_0 = v_N$.  

 The \emph{stable dual graph} of $(X,\frac{1}{m}B)$, denoted $\widetilde{\Gamma}$, is the graph obtained from $\Gamma_{\X}$ by forgetting  the inessential vertices. The lengths of the edges of $\tilde{\Gamma}$ are such that $\tilde{\Gamma}$ and $\Gamma_\X$ are isometric as metric graphs. For an example, see Figure \ref{figStableDualGraph}.

 \begin{figure}
   \begin{tikzpicture}
	\begin{pgfonlayer}{nodelayer}
		\node [style=none] (0) at (-10.25, 6.25) {};
		\node [style=none] (1) at (-8.25, 4.25) {};
		\node [style=none] (2) at (-8.25, 4.75) {};
		\node [style=none] (3) at (-10.25, 2.75) {};
		\node [style=none] (4) at (-10.25, 3.25) {};
		\node [style=none] (5) at (-8.25, 1.25) {};
		\node [style=none] (6) at (-8.25, 1.75) {};
		\node [style=none] (7) at (-10.25, -0.25) {};
		\node [style=none] (8) at (-10.25, 0.25) {};
		\node [style=none] (9) at (-8.25, -1.75) {};
		\node [style=none] (10) at (-10.5, 6.5) {$E_0$};
		\node [style=none] (11) at (-8, 5) {$E_1$};
		\node [style=none] (12) at (-10.5, 3.5) {$E_2$};
		\node [style=none] (13) at (-10.5, 0.5) {$E_4$};
		\node [style=none] (14) at (-8, 2) {$E_3$};
		\node [style=none] (15) at (-8.25, -1.25) {};
		\node [style=none] (16) at (-10.25, -3.25) {};
		\node [style=none] (17) at (-8, -1) {$E_5$};
		\node [style=none] (18) at (-9, -4.25) {$\X_0 = \sum_{i=0}^5 E_i$};
		\node [style=none] (23) at (-1.25, 5.25) {};
		\node [style=none] (24) at (8.75, 5.25) {};
		\node [style=filled] (25) at (0.75, 5.25) {};
		\node [style=filled] (26) at (4.75, 5.25) {};
		\node [style=filled] (27) at (2.75, 5.25) {};
		\node [style=filled] (28) at (6.75, 5.25) {};
		\node [style=filled] (29) at (8.75, 5.25) {};
		\node [style=filled] (30) at (-1.25, 5.25) {};
		\node [style=none] (31) at (-1.25, 4.75) {$v_{E_0}$};
		\node [style=none] (32) at (0.75, 4.75) {$v_{E_1}$};
		\node [style=none] (33) at (2.75, 4.75) {$v_{E_2}$};
		\node [style=none] (34) at (4.75, 4.75) {$v_{E_3}$};
		\node [style=none] (35) at (6.75, 4.75) {$v_{E_4}$};
		\node [style=none] (36) at (8.75, 4.75) {$v_{E_5}$};
		\node [style=none] (37) at (-1.25, -2.25) {};
		\node [style=none] (38) at (8.75, -2.25) {};
		\node [style=filled] (43) at (8.75, -2.25) {};
		\node [style=filled] (44) at (-1.25, -2.25) {};
		\node [style=none] (45) at (-1.25, -2.75) {$v_{E_0}$};
		\node [style=none] (50) at (8.75, -2.75) {$v_{E_5}$};
		\node [style=none] (51) at (-0.25, 5.75) {1};
		\node [style=none] (52) at (1.75, 5.75) {1};
		\node [style=none] (53) at (3.75, 5.75) {1};
		\node [style=none] (54) at (5.75, 5.75) {1};
		\node [style=none] (55) at (7.75, 5.75) {1};
		\node [style=none] (56) at (3.75, 3.75) {$\Gamma_{\X}$};
		\node [style=none] (57) at (3.75, -3.25) {$\tilde{\Gamma}$};
		\node [style=none] (58) at (3.75, -1.75) {5};
	\end{pgfonlayer}
	\begin{pgfonlayer}{edgelayer}
		\draw (0.center) to (1.center);
		\draw (2.center) to (3.center);
		\draw (4.center) to (5.center);
		\draw (7.center) to (6.center);
		\draw (8.center) to (9.center);
		\draw (16.center) to (15.center);
		\draw (23.center) to (24.center);
		\draw (37.center) to (38.center);
	\end{pgfonlayer}
\end{tikzpicture}
  \caption{The above figure shows the stable dual graph in the case when central fiber of the minimal snc model $\X$ of $(X,0)$ is given by $\X_0 = E_0+\dots+E_5$ such that $g(E_0) = g(E_5) = 2$ and $g(E_1) = g(E_2) = g(E_3) = g(E_4) = 0$.}
  \label{figStableDualGraph}
\end{figure}
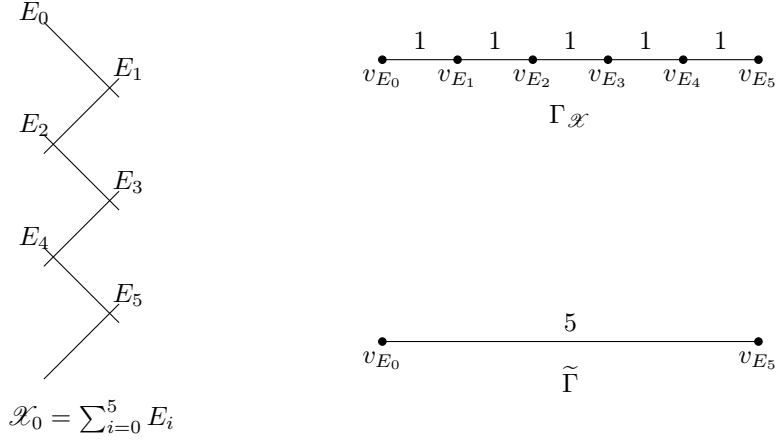

Thus, $V(\tilde{\Gamma})$ is exactly the set of essential vertices of $\Gamma_{\X}$ and an edge in $\tilde{\Gamma}$ corresponds to a maximal inessential chains in  $\Gamma_\X$.

Even though, it is suppressed in the notation, $\tilde{\Gamma}$ depends on the choice of the integer $m$ and the divisor $B$. Note that if $B=0$, then $\tilde{\Gamma}$ is just the dual graph of the stable reduction of $X$ in the sense of \cite{DM69}.

\subsection{The Narasimhan--Simha measure}
\label{subsecPluriBergman}
Let $Y$ be a compact Riemann surface of genus $g$, and let $P_1,\dots,P_r$ be distinct points on $Y$ for some $r \geq 0$. We allow $g = 0$, provided that $r \geq 3$. 

Pick an integer $m \geq 1$ and let $0 < a_1,\dots,a_r < m$ be integers. If $g=0$, we also require that $a_1 + \dots + a_r \geq 2m$. This ensures that $h^0(Y,mK_Y + a_1P_1 + \dots + a_rP_r) > 0$.

Given $\theta, \vartheta \in H^0(Y,mK_Y + a_1P_1 + \dots + a_rP_r)$, we define a volume form $|\theta \wedge \overline{\vartheta} |^{1/m}$ on $Y$ as follows.
Locally if $\theta(z) = f(z)dz^{\otimes m}$ and $\vartheta(z) = g(z)dz^{\otimes m}$, where $f$ and $g$ are local meromorphic functions on $Y$ with poles of order at worst $a_i$ at $P_i$, then, $|\theta \wedge \overline{\vartheta}|^{1/m} = |f(z) \overline{g(z)}|^{1/m} (\frac{i}{2}dz \wedge d\overline{z})$. Also denote $|\theta|^{2/m} := |\theta \wedge \overline{\theta}|^{1/m}$.

Now, we define a continuous function $ \| \cdot \|'_{Y} : H^0(Y,mK_Y + a_1P_1 + \dots + a_rP_r) \to \R_{\geq 0}$ as follows:
$$ \| \theta \|'_{Y} := \left(\int_Y |\theta|^{2/m} \right)^{m/2}.$$

The assumption $a_i < m$ ensures that the integral converges. Note that $\| \cdot \|'_{Y}$ satisfies the following properties (see \cite{NS68} for details):
\begin{itemize}
\item $\| \theta \|'_{Y} = 0 \iff \theta = 0;$
\item $\| \lambda \theta \|'_{Y} = |\lambda| \| \theta \|'_Y \text{ for } \lambda \in \C;$ and
\item $ \{\theta \mid  \|\theta\|'_{Y} = 1 \}$ is a compact subset of $H^0(mK_Y + a_1P_1 + \dots + a_rP_r)$.
\item $\| \cdot \|'_Y$ is not a norm if $m > 1$. 
\end{itemize}

The \emph{Narasimhan--Simha volume form} associated to the line bundle $mK_Y + a_1P_1 + \dots + a_rP_r$ is the continuous positive volume form on $Y$ defined by
$$ \tau^{(m,a_1P_1+\dots+a_rP_r)}(z) = \sup_{ \| \theta \|'_Y = 1} |\theta|^{2/m}(z),$$
where the supremum is over $\{ \theta \in H^0(mK_Y + a_1P_1 + \dots + a_rP_r)   \mid \| \theta \|'_Y = 1 \}$. Since $|\theta|^{2/m}(z)$ is a real cotangent vector at $z$, it lies in ordered set $\R_{\geq 0} \cup \{\infty\}$ and the supremum makes sense. Since the supremum is over the compact set $\{ \theta \mid \| \theta \|'_Y = 1 \}$, the supremum is indeed a maximum. If no confusion arises, we skip the superscript and just use denote $\tau$ to denote $\tau^{(m,a_1P_1+\dots+a_rP_r)}$. We could think of $\tau$ as a continuous section of the real line bundle $|K_Y|^2 \otimes |\O_Y(P_1)|^{2a_1/m} \otimes \dots \otimes |\O_Y(P_r)|^{2a_r/m}$ on $Y$.

If $m \geq 2$, then note that $mK_Y + a_1 P_1 + \dots + a_rP_r$ is base point free and thus, $\tau$ does not vanish anywhere on $Y$.
Also note that $\tau(z) < \infty$ if $z \in Y \setminus \{ P_1,\dots,P_r \}$ and that the total mass of $\tau$ is finite, but does not seem to be easy to calculate. Since the line bundle $mK_Y + a_1 P_1 + \dots + a_rP_r$ is base point free, given a $P_i$, there exists a global section of $mK_Y + a_1 P_1 + \dots + a_rP_r$ that looks locally like $z^{-a_i}dz^{\otimes m}$ near $P_i$. Thus, locally near $P_i$, $\tau$ is given by $\phi \cdot |z|^{-2a_i/m}(\frac{i}{2}dz \wedge d\zbar)$, where $\phi$ is a continuous function in a neighborhood of $P_i$.

Note that if $f : Y \to Y$ is a biholomorphism fixing $P_1,\dots,P_r$, then the pushforward measure $f_*\tau$ is equal to $\tau$ i.e.~$\tau$ is invariant under the action of an automorphism of the marked curve $(Y;P_1,\dots,P_r)$. 

\subsection{The pluri-Bergman measure}
Let the notation be as in the previous subsection. We define a Hermitian pairing on $H^0(mK_Y + a_1P_1 + \dots + a_rP_r)$ as follows \cite{NS68}:
\begin{equation}
\label{eqnGenBergmanPairing}
\langle \theta, \vartheta \rangle = \left(\frac{i}{2}\right)^m \int_Y  \frac{\theta \wedge \overline{\vartheta}}{\tau^{m-1}}.  
\end{equation}

In the above $(\frac{i}{2})^m\frac{\theta \wedge \overline{\vartheta}}{\tau^{m-1}}$ is the (1,1)-form given as follows. If $\theta = f(z)dz^{\otimes m}$, $\vartheta = g(z)dz^{\otimes m}$ and $\tau = h(z)(\frac{i}{2}dz\wedge d\zbar)$ locally for some holomorphic functions $f,g$ and positive real-valued function $h$, then $(\frac{i}{2})^m\frac{\theta \wedge \overline{\vartheta}}{\tau^{m-1}} = \frac{f(z)\overline{g(z)}}{h(z)^{m-1}}(\frac{i}{2}dz\wedge d\zbar)$ locally.
We also use the notation $\frac{|\theta \wedge \overline{\vartheta}|}{\tau^{m-1}}$ to denote $\frac{|f(z)\overline{g(z)}|}{h(z)^{m-1}}(\frac{i}{2}dz\wedge d\zbar)$ locally.

The continuous $(1,1)$-form $\tau$ does not vanish anywhere on $Y$ and thus the integral is well defined and finite. Note that if $m = 1$, then $\tau$ does not play any role and we recover the Hermitian pairing induced by the Bergman metric.

Let $e_1,\dots,e_M$ be an orthonormal basis of $H^0(mK_Y + a_1P_1 + \dots + a_rP_r)$ with respect to the above pairing. Using elementary linear algebra, we see that the $(1,1)$-form 
$$ \mu^{(m,a_1P_1+\dots+a_rP_r)} = \sum_{i=1}^M \frac{|e_i \wedge \overline{e_i}|}{\tau^{m-1}}$$
does not depend on the choice of the orthonormal basis. We call the corresponding Radon measure on $Y$ as the \emph{pluri-Bergman measure} on $Y$ induced by line bundle $mK_Y + a_1P_1 + \dots + a_rP_r$. Whenever there is no confusion regarding the line bundle, we skip the superscript and denote $\mu = \mu^{(m,a_1P_1+\dots+a_rP_r)}$. It is also given by the formula

$$ \mu = \sup \left\{ \frac{|\theta \wedge \overline{\theta}|}{\tau^{m-1}} \ \Big| \ \theta \in H^0(mK_Y + a_1P_1 + \dots + a_rP_r), \int_{Y}  \frac{|\theta \wedge \overline{\theta}|}{\tau^{m-1}} = 1 \right\},$$
which is proved using the same arguments as in the proof of Propositions 1.1 and 1.2 of \cite{Ber10}. It also follows from this description that in the case when $m=1$ and $B=0$, $\mu^{(1)}$ is the same as $\tau^{(1)}$ and is the Bergman measure on $Y$. 

Since $\tau$ does not vanish on $Y$, $\mu(z)$ is finite for all points $ z \in Y \setminus \{ P_1,\dots,P_r \}$. From the second description of $\mu$, it follows that $\mu$ is nowhere vanishing. By using the fact that there is a global section of $mK_Y + a_1P_1 + \dots + a_rP_r$ that looks like $z^{-a_i}dz^{\otimes m}$ near $P_i$ and that $\tau \sim C_1|z|^{-2a_i/m}(\frac{i}{2} dz \wedge d\zbar)$ near $P_i$, we conclude that $\mu \sim C_2|z|^{-2a_i/m}(\frac{i}{2} dz \wedge d\zbar)$ near $P_i$. Note that the total mass of $\mu$ is just $h^0(mK_Y + a_1P_1 + \dots + a_rP_r)$.

\subsection{The dualizing sheaf and its tensor powers}
\label{subsecDualizingSheaf}
Let $\X'$ denote an snc model of $X$ such that $\X_0'$ is reduced. Then, the exists a dualizing sheaf $\omega_{\X'_0} = \omega_{\X'}(\X'_0)|_{\X'_0}$ on $\X'_0$. 

Let $E_1$ and $E_2$ be irreducible components of $\X_0'$. A local section of $\omega_{\X_0'}$ near a node $P = E_1 \cap E_2$ is the data of meromorphic one-forms $f_1,f_2$ on $E_1$ and $E_2$, respectively, with at worst simple-pole along $P$, such that the residues of $f_1$ and $f_2$ at $P$ sum to 0 \cite[Section I]{DM69}.

Using this local description of sections of $\omega_{\X'_{0}}$ , we also get the following description of local sections of $\omega_{\X'_{0}}^{\otimes k}$, where $k$ is an integer (possibly negative). If $k \geq 1$, let us denote $dz^{\otimes k} = dz \otimes \dots \otimes dz$; if $k$ is negative, we can think of $dz^{\otimes k}$ as a formal symbol satisfying the appropriate change of coordinates. Then, a local section $\theta$ of $\omega_{\X'_{0}}^{\otimes k}$ near $P = E_1 \cap E_2$ is just given by the data of two meromorphic $k$-canonical forms $f(z)dz^{\otimes k}$ and $g(w)dw^{ \otimes k}$ locally on $E_2$ and $E_1$ near $P$, respectively, such that
\begin{itemize}
\item  $f$ and $g$ can have at worst poles of order $k$ at the origin. (When $k$ is negative, this means that $f$ and $g$ vanish to order at least $-k$ at the origin)
\item If we write $f(z) = \sum_{n \geq -k}a_{n}z^n$ and $g(w) = \sum_{n \geq -k}b_nw^{n}$ locally around $P$, then $$a_{-k} +  (-1)^{k+1}b_{-k} = 0.$$
We call $a_{-k}$ and $b_{-k}$ the \emph{residues} of $\theta$ at $P$ along $E_2$ and $E_1$ respectively. 
\end{itemize}

Let $E_1,\dots,E_s$ denote the irreducible components of $\X'_{0}$. Let $P_{1}^{(i)},\dots,P_{r_i}^{(i)}$ denote the nodal points of $\X'_{0}$ that lie in $E_i$.

Now pick an orientation on $\Gamma_{\X'}$ i.e.~for every edge $e \in E(\Gamma_{\X'})$, we pick a direction. Let $e^{-}$ and $e^{+}$ denote the initial and final vertex of $e$ with respect to the chosen orientation. Summarizing the above discussion, we have a short exact sequence of sheaves on $\X'_{0}$. 
\begin{multline}
\label{eqnSESOmegaN0}
0 \to \omega_{\X_{0}}^{\otimes k} \to \bigoplus_i \O_{E_i}(kK_{E_i} + kP^{(i)}_1 + \dots + kP^{(i)}_{r_i}) \\ \to \bigoplus_{P \in \X_0 \text{ node }}\C(P) \to 0,
\end{multline}
where the first map is given by the restrictions and the second map is given by   taking the sum (respectively difference) of residues if $k$ is odd (respectively even) i.e.~this map is given by $(\psi_i)_{i} \mapsto (\res_{P}(\psi_{e^+_P}) + (-1)^{k+1} \res_{P}(\psi_{e^-_P}))_P$. 

The following short exact sequence will also be useful. 
\begin{multline}
  \label{eqnSESOmegaN20}
  0 \to \bigoplus_i \O_{E_i}(kK_{E_i} + (k-1)P^{(i)}_1 + \dots + (k-1)P^{(i)}_{r_i}) \\ \to \omega_{\X_{0}}^{\otimes k} \to  \bigoplus_{P \in \X_0 \text{ node }}\C(P) \to 0,
\end{multline}
where the first map exists because all the residues of sections of the left term are zero, so there is no compatibility of residues to be checked. The second map is taking the residue at each node $P$ along the irreducible component associated to $e_P^-$.

\section{Metrized curve complex hybrid space}
\label{secMCCHS}
We describe the metrized curve complex and the associated hybrid space in this section. See \cite{AB15} and \cite[Section 7]{Shi20A} for more details. In this section $\X$ will denote an arbitrary snc model of $X$. We do not assume that $\X$ is semistable and we will not keep track of the divisor $B$. 

\subsection{Metrized curve complex}
Let $\tilde{\X_0}$ denote the normalization of $\X_0$ i.e.~the disjoint union of all irreducible components of $\X_0$. The \emph{metrized curve complex} $\Delta_\CC(\X)$ associated to $\X$ is a topological space defined as follows.
$$ \Delta_\CC(\X) = \left(\tilde{\X_0} \sqcup \bigsqcup_{e \in E(\Gamma_\X)}[0,l_e]\right) \Big/\sim, $$
where $\sim$ is the identification of the end points $[0,l_{e}]$ with the corresponding points that lie over the nodal point associated to $e$. Recall that $l_e$ is the length of the edge $e \in E(\Gamma_\X)$. See Figure \ref{figExampleGenus4Family} for an example.

\subsection{Metrized curve complex hybrid space}
As a set, the metrized curve complex hybrid space, $\X^\hyb_\CC$, is given by
$$ \X^\hyb_\CC = X \sqcup \Delta_\CC(\X).$$
We also have a map $\X^\hyb_\CC \to \D$ given by extending the map $X \to \D^*$ and sending $\Delta_\CC(\X)$ to the origin. This map will turn out to be continuous in the topology on $\X^\hyb_{\CC}$. Before we describe the topology on $\X_\CC^\hyb$, we make a few definitions. 

First, consider a point $Q \in E$, where $E \subset \X_0$ is an irreducible component of multiplicity $a$ such that $Q$ is not a nodal point. We can find an open set $U_1 \subset \X$ containing $Q$ with coordinates $z,w$ on $U_1$ such that $U_1 \cap \X_0 = U_1 \cap E$, $|z|,|w| < 1$ and the map to $\D$ is given by $(z,w) \mapsto z^a$. We say that $(U_1,z,w)$ is a coordinate chart \emph{adapted} to the irreducible component $E$ and centered at $Q$.

Now, consider a nodal point $P = E_1 \cap E_2$ in $\X_0$, where $E_1,E_2 \subset \X_0$ are irreducible components of multiplicity $a,b$ respectively. Then, we can find an open chart $U_2$ and coordinates $z,w$ on $U_2$ such that $U_2 \cap \X_0 = U_2 \cap (E_1 \cup E_2)$, $|z|,|w| < 1$ and the map to $\D$ is given by $(z,w) \mapsto z^aw^b$. We say that $(U_2,z,w)$ is adapted to the node $P = E_1 \cap E_2$. 

To describe the topology on $\X^\hyb_\CC$, it is enough to describe the neighborhood basis of each point.  
\begin{itemize}
\item Firstly, we require that $X \to \X_\CC^\hyb$ is an open immersion. This describes the neighborhood basis of points in $X$.
\item Pick a point $Q \in E$, where $E \subset \X_0$ is an irreducible component and $Q$ is not a nodal point on $\X_0$. Let $(U_1,z,w)$ be a coordinate chart adapted to $E$ and centered at $Q$. Viewing $U_1$ as a subset of $\X^\hyb_\CC$, we get a neighborhood of $Q$. By shrinking such adapted coordinate charts, we get a neighborhood basis of $Q$.
\item Pick a point $Q \in e_P$, where $P = E_1 \cap E_2$ is a nodal point in $\X_0$, where $E_1,E_2$ have multiplicities $a,b$ respectively such that $Q$ does not lie in $\tilde{\X}_0$. Identify $e_P$ with $[0,\frac{1}{ab}]$, where $0$ gets identified with $v_{E_1}$ and $\frac{1}{ab}$ with $v_{E_2}$. 
  Let $(U_2,z,w)$ be a coordinate chart adapted to the node $P = E_1 \cap E_2$.  Pick $\alpha,\beta$ so that $0 < \alpha < Q < \beta < \frac{1}{ab}$. Then,
  $$ \left\{ (z,w) \in U_2 \mid \alpha < \frac{\log|w|}{a\log|z^aw^b|} < \beta \right\} \cup (\alpha,\beta)$$
  is a neighborhood of $Q$. Shrinking $U_2$ and letting $\alpha,\beta \to Q$, we get a neighborhood basis of $Q$.
  \item Pick a point $Q = e_P \cap E_1$, where $P = E_1 \cap E_2$ is a nodal point in $\X_0$, where $E_1,E_2$ have multiplicities $a,b$ respectively. Identify $e_P$ with $[0,\frac{1}{ab}]$, where $0$ gets identified with $v_{E_1}$ and $\frac{1}{ab}$ with $v_{E_2}$. 
  Let $(U_2,z,w)$ be a coordinate chart adapted to the node $P = E_1 \cap E_2$.  Pick  $0 < \epsilon \ll \frac{1}{ab}$. Then,
  $$ \left\{ (z,w) \in U_2 \mid  \frac{\log|w|}{a\log|z^aw^b|} < \epsilon \right\} \cup (U_2 \cap E_1) \cup [0,\epsilon)$$
 is a neighborhood of $Q$. Shrinking $U_2$ and letting $\epsilon \to 0$, we get a neighborhood basis of $Q$. 
\end{itemize}

Alternatively, we can define the topology on $\X^\hyb_\CC$ to be the coarsest topology for which the maps $\X^\hyb_\CC \to \X$ and $\X^\hyb_\CC \to \X^\hyb$ are continuous, where $\X^\hyb = X \cup \Gamma_\X$ denotes the Boucksom-Jonsson hybrid space \cite[Section 7]{Shi20A}. 

\begin{lem}
If $\X,\X'$ are models of $X$ such that $\X'$ dominates $\X$ i.e.~a proper map $\X' \to \X$ which restricts to identity on $X$, then there exists a unique continuous surjective map $(\X')^\hyb_\CC \to \X^\hyb_\CC$ that restricts to identity on $X$. 
\end{lem}
\begin{proof}
  Such a map $\X' \to \X$ is given by a composition of blowups along closed points in the central fiber \cite[Theorem 1.15]{Lic68}, we may reduce to the case when $\X' \to \X$ is obtained by a single blowup. If $\X' \to \X$ is obtained by blowing up a smooth point in $\X_0$, then we get a map $\Delta_\CC(\X') \to \Delta_\CC(\X)$ obtained by collapsing the extra edge and curve in $\Delta_\CC(\X')$ to the center of the blowup.

  If $\X' \to \X$ is obtained by blowing up a nodal point in $\X_0$, then $\Delta_\CC(\X') \to \Delta_\CC(\X)$ is obtained by collapsing the extra curve in $\Delta_\CC(\X')$  to the corresponding point in $\Gamma_\X$. More precisely, suppose $\X'$ is obtained by blowing up $P = E_1 \cap E_2$, where $E_1$ and $E_2$ are irreducible components of multiplicity $a$ and $b$ respectively and let $E$ be the exceptional divisor of the blowup. We can identify $e_{P} \simeq [0,\frac{1}{ab}]$, where $v_{E_1}$ is identified with 0 and $v_{E_2}$ is identified with $\frac{1}{ab}$. The exceptional curve $E$ is collapsed to the point $\frac{1}{a(a+b)} \in e_P$ and the edges $e_{E_1 \cap E}$ and $e_{E_2 \cap E}$ are identified with $[0,\frac{1}{a(a+b)}]$ and $[\frac{1}{a(a+b)}, \frac{1}{ab}]$. 

  In both cases, we get a  continuous surjective map $\Delta_\CC(\X') \to \Delta_\CC(\X)$, which gives rise to a surjective map $(\X')^\hyb_\CC \to \X^\hyb_\CC$ 
  
To show that this map  is continuous, it is enough to note that the compositions $(\X')_\CC^\hyb \to \X' \to \X$ and $(\X')^\hyb_\CC \to (\X')^\hyb \to \X^\hyb$ are continuous and that these compositions are the same as the compositions $(\X')_\CC^\hyb \to \X^\hyb_\CC \to \X$ and $(\X')^\hyb_\CC \to \X_\CC^\hyb \to \X^\hyb$.  
\end{proof}

\subsection{Convergence of measures on the metrized curve  complex hybrid space}
We outline some general techniques that will be used to prove Theorem \ref{mainThmA} and \ref{mainThmB}. 
\begin{lem}
\label{lemConvergenceMCCHC}
Let $\X$ be an snc model of $X$. Let $(\nu_t)_{t \in \D^*}$ be a family of Radon measures on $X$ with the support of $\nu_t$ contained in $X_t$ and such that $\limsup_{t \to 0} \nu_t(X_t) < \infty$. Let $\nu_0$ be a Radon measure on $\Delta_\CC(\X)$. To show that $\nu_t \to \nu_0$ weakly as measures on $\X^\hyb_\CC$, it is enough to prove the following.
\begin{enumerate}
\item Let $(U_1,z,w)$ be a coordinate chart adapted to an irreducible component $E \subset \X_0$ and let $f$ be a continuous function on $U_1$ with compact support. Then,
  $$\int_{U_1 \cap X_t} f \nu_t \to \int_{U_1 \cap E} f \nu_0$$
  as $t \to 0$.
\item Let $(U_2,z,w)$ be a coordinate chart adapted to a node $P = E_1 \cap E_2$, where $E_1, E_2$ have multiplicities $a,b$ in $\X_0$ respectively. Let $0 < \alpha < \beta \leq \frac{1}{2ab}$ and let $f$ be a continuous function on $[0,\frac{1}{ab}] \simeq e_P$. Then,
  $$\int_{\{w \in U_2 \mid \alpha \leq \frac{\log|w|}{a\log|t|} \leq \beta\}} f\left(\frac{\log|w|}{a \log|t|} \right) \nu_t \to \int_{[\alpha,\beta]} f\nu_0$$
  as $t \to 0$.
\item Let $(U_2,z,w)$ be a coordinate chart adapted to a node $P = E_1 \cap E_2$, where $E_1, E_2$ have multiplicities $a,b$ in $\X_0$ respectively. Let $0 < \epsilon < \frac{1}{2ab}$ and identify $e_P \simeq [0,\frac{1}{ab}]$. Let $$D_\epsilon = \{ (w,u) \in \D \times [0,\epsilon) \mid \text{Either } w = 0 \text{ or } u = 0 \}$$ and let $r : \D \times [0,\epsilon) \to D_\epsilon$ be a strong deformation retract. We can identify $D_\epsilon$ with $(U_2 \cap E_1) \cup [0,\epsilon)$. Let $f$ be a compactly supported continuous function on $D_\epsilon$. Then,
  $$ \int_{\{w \in U_2 \mid \frac{\log|w|}{a\log|t|} < \epsilon \}} f\left(r\left( w,\frac{\log|w|}{a\log|t|}\right)\right) \nu_t \to \int_{D_\epsilon} f \nu_0 $$
  as $t \to 0$.
\end{enumerate}
\end{lem}
\begin{proof}
  Let $h$ be a continuous function in a neighborhood of $\Delta(\X)$. We need to show that $\int_{X_t} h \nu_t \to \int_{\Delta_\CC(\X)} h \nu_0$.
  
  The sets listed in the lemma form the neighborhood basis of points in $\Delta_\CC(\X)$ in $\X_\CC^\hyb$. So, we can cover a neighborhood of $\Delta_\CC(\X)$ using finitely many open sets of these forms. Now consider a partition of unity $\{\chi_i \}_i$ adapted to such a cover. Writing $h = \sum_i \chi_i h$, it is enough to show that $\int \chi_i h \nu_t \to \int \chi_i h \nu_0$ as $t \to 0$. So, we are reduced to the case where $h$ is supported in a set of one of the forms listed above.

  If $h$ is supported in the set listed in (1), then take $f = h$ and  there is nothing to show.

  If $h$ is supported in the set listed in (2), let $f = h|_{e_P}$. Then, $h - f(\frac{\log|w|}{a \log|t|})$ is a compactly supported continuous function which vanishes along $\Delta_\CC(\X_0)$. Thus, given $\epsilon' > 0$, we can find $t_0 > 0$ such that $|h - f(\frac{\log|w|}{a \log|t|})| < \epsilon'$ on $X_t$ for all $|t| < t_0$. Thus, we get that $\int |h - f(\frac{\log|w|}{a \log|t|})| \nu_t < \epsilon' \nu_t(X_t) $. Letting $\epsilon' \to 0$ and using the fact that $\limsup_{t \to 0}\nu_t(X_t) < \infty$, we get that
  $$\int_{\{w \in U_2 \mid \alpha \leq \frac{\log|w|}{a\log|t|} \leq \beta\}} \left|h - f\left(\frac{\log|w|}{a \log|t|}\right)\right| \nu_t \to 0,$$
  and thus
  $$
  \lim_{t \to 0} \int h \nu_t = \lim_{t \to 0} \int_{\{w \in U_2 \mid \alpha \leq \frac{\log|w|}{a\log|t|} \leq \beta\}} f\left(\frac{\log|w|}{a \log|t|}\right) \nu_t = \int_{[\alpha,\beta]}  f \nu_0 = \int h \nu_0.$$

  A similar argument also shows that if $h$ is supported in the set listed in (3), then $\int h \nu_t \to \int h \nu_0$.
\end{proof}
\subsection{Extending convergence to higher models}
Let $\X$ and $\X'$ be models of $X$ such that $\X'$ dominates $\X$. Recall that this gives rise to a unique continuous surjective map ${\X'}_\CC^{\hyb} \to  \X_\CC^\hyb$.

\begin{lem}
\label{lemLimitOnOnlyOneModel}  
Let $\X, \X'$ be snc models of $X$ such that $\X'$ dominates $\X$. Let $(\nu_t)_{t \in \D^*}$ be a family of Radon measures on $X_t$ and $\nu_0$ a Radon measure on $\Delta_\CC(\X)$ such that $\nu_t$ converges weakly to $\nu_0$ on $\X_\CC^\hyb$. Suppose that $\nu_0(\{Q\}) = 0$ for all points $Q \in \Delta_\CC(\X)$. Then there exists a unique measure $\nu_0'$ on $\Delta_\CC(\X')$ such that $\nu_t \to \nu_0'$ weakly on ${(\X')}^\hyb_\CC$.
\end{lem}
\begin{proof}
  Since the map $q \colon \X' \to \X$ is a composition of blowups \cite[Theorem 1.15]{Lic68}, we see that the map $\X_0' \to \X_0$ is obtained by contracting some irreducible components of $\X_0'$ and that the map $\Delta_\CC(\X') \to \Delta_\CC(\X)$ is obtained by collapsing some curves and edges to points. Let $q_\CC \colon \Delta_\CC(\X') \to \Delta_\CC(\X)$, $q^\hyb_\CC \colon {(\X')}_\CC^\hyb \to \X_\CC^\hyb$ denote the induced maps.

  
We claim that there is a unique measure $\nu_0'$ on $\Delta_\CC(\X')$ such that the pushforward measure $(q_\CC)_* (\nu_0')$ is $\nu_0$. This is easy to see as there exist finitely many points $Q_1,\dots Q_s \in \Delta_\CC(\X)$ such that $q_\CC$ is an homeomorphism over $\Delta_\CC(\X) \setminus \{ Q_1,\dots,Q_s \}$. This determines $\nu_0'|_{q_\CC^{-1}(\Delta_\CC(\X) \setminus \{ Q_1,\dots,Q_s \})}$ and since $\nu_0(\{Q_i\}) = 0$, we also get that $\nu'_0|_{q_\CC^{-1}(Q_i)}= 0$ for all $i = 1,\dots,s$.

  Pick $0 < r < 1$. Let ${(\X')}^\hyb_{\CC,r}$ denote the preimage of $r\overline{\D}$ under the map $\pi\colon {(\X')}^\hyb_{\CC,r} \to \D$. Then, ${(\X')}^\hyb_{\CC,r}$ is a compact topological space and $\nu_t$ for $|t| \leq r$ is a collection of Radon measures on ${(\X')}^\hyb_{\CC,r}$. Since $\nu_t(X_t) \to \nu_0(\Delta_\CC(\X))$, we may decrease $r$ to further assume that $\nu_t(X_t) \leq \nu_0(\Delta_\CC(\X)) + 1$ for all $t \in r\D^*$.

  Let $t_1,t_2,\dots$ be a sequence in $r\overline{\D}$ that converges to 0. 
Applying the Banach-Alaoglu theorem to the dual space of continuous functions on $(\X')^\hyb_{\CC,r}$, we get that, after passing to a subsequence, the measures $\nu_{t_{i_k}}$ has a weak limit $\tilde{\nu}_0$. Then, we get that $\nu_{t_{i_k}} \to \tilde{\nu_0}$ on ${(\X')}_{\CC,r}^{\hyb}$. But since pushforward of Radon measures under a continuous map commutes with taking weak limits, we get that $\nu_{t_{i_k}} \to (q_\CC)_*(\tilde{\nu}_0)$. But this means that $(q_\CC)_*(\tilde{\nu}_0) = \nu_0$. By the uniqueness of such a measure we get that $\tilde{\nu}_0 = \nu_0'$ i.e. all convergent subsequnces have the same weak limit. Thus, we get that $\nu_t \to \nu'_0$ on ${(\X')}_{\CC,r}^{\hyb}$ and hence on ${(\X')}_{\CC}^{\hyb}$. 
\end{proof}

\begin{cor}
\label{corConvergenceBerkovichHybridSpace}
Let $X$, $\X$, $(\nu_t)_{t \in \D^*}$, $\nu_0$ be as in Lemma  \ref{lemLimitOnOnlyOneModel}. Then, there exists a Radon measure $\nu_0'$ on $X_\Ct^\an$ such that $\nu_t \to \nu_0'$ weakly as measures on $X^\hyb$. Moreover, the support of $\nu_0'$ lies in contained in the skeletal subset $\Gamma_\X \subset X^\an_{\Ct}$. 
\end{cor}
\begin{proof}
  Recall that the Berkovich hybrid space $X^\hyb = X \cup X^\an_\Ct$ can also be obtained as an inverse limit of the Boucksom-Jonsson hybrid spaces $(\X')^\hyb = X \cup \Gamma_{\X'}$, where $\X'$ runs through all snc models of $X$; these form a directed system. Therefore, to prove convergence on $X^\hyb$, it is enough to prove a compatible convergence of $\nu_t$ on $(\X')^\hyb$ for all snc models $\X'$ of $X$.

  Since the collection of models that dominate $\X$ form a cofinal system, it is enough to prove this for models $\X'$ that dominate $\X$. Consider such a model $\X'$. From Lemma \ref{lemLimitOnOnlyOneModel}, we get that the limit of $\nu_t$ on $(\X')_\CC^\hyb$ exists. From the continuous map $(\X')_\CC^\hyb \to (\X')^\hyb$ obtained by collapsing the curves in the central fiber, we see that the limit of $\nu_t$ on $(\X')^\hyb$ is just the pushforward of the limit of $\nu_t$ on $(\X')_\CC^\hyb$. Using the techniques from the proof of Lemma \ref{lemLimitOnOnlyOneModel}, it is also easy to check that the limits are compatible i.e.~if $\X'$ and $\X''$ are snc models of $X$ such that $\X''$ dominates $\X'$, then limit on $\Gamma_{\X'}$ is just the pushforward under the retraction map $\Gamma_{\X''} \to \Gamma_{\X'}$.
  This proves that the limit of $\nu_t$ on $X^\hyb$ exists. 

  The statement about the support follows from the fact that the limit of $\nu_t$ on $(\X')^\hyb_\CC$ does contain any of the edges collapsed under the map $\Delta_\CC(\X') \to \Delta_\CC(\X^\hyb)$.  
\end{proof}

\section{The sheaf $\omega_{\X_{0}}^{\otimes m}(\B|_{\X_0})$}
\label{secSectionsOfOmegaMB}
Throughout this section, let $m \geq 2$ denote a positive integer. Let $(X,B)$ satisfy the assumptions listed in Section \ref{subsecModels}. Let
$\X$ denote the minimal snc model of $(X,\frac{1}{m}B)$. Recall that $\X_0$ is reduced and that $\X_0$ and $\B$ do not intersect at nodal points in $\X_0$. In this section, we give a description of the global sections of $\omega_{\X_{0}}^{\otimes m}(\B|_{\X_0})$.

\subsection{Local description}
To get a local description of sections of $\omega_{\X_0}^{\otimes m}(\B|_{\X_0})$,  we just need to tensor the short exact sequence \eqref{eqnSESOmegaN0} for $\X_0' = \X_0$ and $k = m$ with $\O_{\X_0}(\B|_{\X_0})$ to get 
\begin{multline}
\label{eqnSESOmegaN}
0 \to \omega_{\X_{0}}^{\otimes m}(\B|_{\X_0}) \to \bigoplus_i \O_{E_i}(mK_{E_i} + mP^{(i)}_1 + \dots + mP^{(i)}_{r_i} + \B|_{E_i}) \\ \to \bigoplus_{P \in \X_0 \text{ node }}\C(P) \to 0,
\end{multline}
where the first map is given by the restrictions and the second map is given by   taking the sum/difference of residues (See Section \ref{subsecDualizingSheaf} for details). 

\subsection{Dimension of global sections}
To understand the convergence of the \break Narasimhan--Simha, we would like to use Grauert's Lemma \cite[Corollary III.12.9]{Har77} to be able to conclude that there exists an open neighborhood $U \subset \X$ of $\X_0$ such that we can find $\theta_1,\dots,\theta_M \in H^0(U,\omega_{\X}^{\otimes m}(\B))$ such that $\theta_i|_{X_t}$ is a basis of $H^0(X_t, \omega_{X_t}^{\otimes m}(B|_{X_t}))$ for $|t| \ll 1$ and $\theta_i|_{\X_{0}}$ is a basis for $H^0(\X_{0},\omega_{\X_{0}}^{\otimes m}(\B|_{\X_0}))$. To do this, it is enough to show that $h^0(\X_{0},\omega_{\X_{0}}^{\otimes m}(\B|_{\X_0})) = h^0(X_t,\omega_{X_t}^{\otimes m}(B|_{X_t})) = (2m-1)(g-1) + \deg(B|_{X_t})$.

\begin{rmk}
The reason for working with the minimal snc model of $(X,\frac{1}{m}B)$ is precisely because $h^0(\X_{0},\omega_{\X_{0}}^{\otimes m}(B|_{\X_0})) = h^0(X_t,\omega_{X_t}^{\otimes m}(B|_{X_t}))$ is satisfied for the minimal snc model $\X$, while it  not necessarily satisfied by a general snc model. The minimality assumption plays a role in the proof of Lemma \ref{lemH1Vanishing}, where it helps us control the  $H^0(E_i,(1-m)(K_{E_i} + P^{(i)}_1 + \dots + P^{(i)}_{r_i})-\B|_{E_i})$ term that shows up.    
\end{rmk}

\begin{lem}
\label{lemDimH0OmegaN}
  $$h^0(\X_{0}, \omega^{\otimes m}_{\X_0}(\B|_{\X_0})) = (2m-1)(g-1) + \deg(\B|_{\X_0}).$$
\end{lem}
\begin{proof}
  Using the short exact sequence \eqref{eqnSESOmegaN}, we get
  $$ \chi(\omega_{\X_{0}}^{\otimes m}(\B|_{\X_0})) = \left(\sum_i\chi(\O_{E_i}(mK_{E_i} + mP^{(i)}_1 + \dots + mP^{(i)}_{r_i}+\B|_{E_i})\right) - \#E(\Gamma_{\X}),$$
  where $\chi(\F) = h^0(\F) - h^1(\F)$ is the Euler characteristic of a sheaf $\F$.
  By Riemann--Roch, the right hand side is
\begin{align*}
 & \left(\sum_i\chi(\O_{E_i}(mK_{E_i} + mP^{(i)}_1 + \dots + mP^{(i)}_{r_i} + \B|_{E_i})\right) - \#E(\Gamma_{\X}) \\
 &= \left( \sum_i((2g(E_i)-2+r_i)m + \deg(B|_{E_i}) - g(E_i) +1) \right) - \#E(\Gamma)
  \\&= (2m-1)\left(\sum_{i} g(E_i)\right) + (-2m+1) \# V(\Gamma) + (2m-1) \# E(\Gamma) + \deg(\B|_{\X_0})
  \\ &= (2m-1)\left(\sum_{i} g(E_i) + g(\Gamma_\X) - 1 \right) + \deg(\B|_{\X_0})
  \\ &= (2m-1)(g-1) + \deg(\B|_{\X_0}).
\end{align*}
Therefore the result follows if we can show that $h^1(\omega_{\X_{0}}^{\otimes m}(\B|_{\X_0})) = 0$. This is proved in the following lemma. 
\end{proof}

\begin{lem}
 \label{lemH1Vanishing}
Let $\X$ be the minimal snc model of $(X,\frac{1}{m}B)$. Then, $$h^1(\X_{0}, \omega_{\X_{0}}^{\otimes m}(\B|_{\X_0})) = 0.$$
\end{lem}
\begin{proof}
  Using Serre duality, we have $h^1(\omega_{\X_{0}}^{\otimes m}(\B|_{\X_0})) = h^0(\omega_{\X_{0}}^{\otimes 1-m}(-\B|_{\X_0}))$ and it is enough to show that $h^0(\X_{0}, \omega_{\X_{0}}^{\otimes 1-m}(-\B|_{\X_0})) = 0$.

  Consider the short exact sequence obtained by tensoring \eqref{eqnSESOmegaN0} for $\X_0' = \X_0$ and $k = 1-m$ with $\O_{\X_0}(-\B|_{\X_0})$.

  \begin{multline*}
0 \to \omega_{\X_{0}}^{\otimes(1-m)}(-\B|_{\X_0}) \to \bigoplus_i \O_{E_i}((1-m)(K_{E_i} + P^{(i)}_1 + \dots + P^{(i)}_{r_i}) - \B|_{E_i}) \\ \to \bigoplus_{P \in \X_0 \text{ node }}\C(P) \to 0,
\end{multline*}

  By considering the long exact sequence induced in cohomology, we get
  \begin{multline*}
  0 \to H^0(\omega_{\X_{0}}^{\otimes 1-m}(-\B|_{\X_0})) \to \bigoplus_i H^0(E_i,(1-m)(K_{E_i} + P^{(i)}_1 + \dots + P^{(i)}_{r_i})-\B|_{E_i}) \to
   \\
   \bigoplus_{P \subset \X_0 \text{ node }}\C(P).  
  \end{multline*}
  Since $m \geq 2$, $H^0(E_i,(1-m)(K_{E_i} + P^{(i)}_1 + \dots + P^{(i)}_{r_i})-B|_{E_i}) = 0$ in any one of the following cases.
\begin{itemize}
  \item $g(E_i) \geq 2$,
  \item $g(E_i) = 1$ and $\val(E_i) \geq 1$
  \item $g(E_i) = 1$ and $\deg(B|_{E_i}) \geq 1$
  \item $g(E_i) = 0$ and $\val(E_i) \geq 3$
  \item $g(E_i) = 0$, $\val(E_i)  = 2$ and $\deg(B|_{E_i}) \geq 1$, 
  \item $g(E_i) = 0$, $\val(E_i) = 1$ and $\deg(B|_{E_i}) \geq m$, or
  \item $g(E_i) = 0$, $\val(E_i) = 0$ and $\deg(B|_{E_i}) \geq 2m-1$.
\end{itemize}
  
Comparing this with all the constraints on the irreducible components of $\X_0$ mentioned in Section \ref{subsecModels}, we see that the only contribution in the middle term comes from the inessential components i.e.~the components for which $g(E_i) = \deg(\B|_{E_i}) = 0$ and $\val(E_i) = 2$. Note that here we crucially use that $\X$ is the minimal snc model of $(X,\frac{1}{m}B)$. 
In this case, 
$$ h^0(E_i,(1-m)(K_{E_i} + P^{(i)}_1 + P^{(i)}_{2})) = h^0(\P^1,\O_{\P^1}) = 1,$$
and any section of $H^0(E_i,(1-m)(K_{E_i} + P^{(i)}_1 + P^{(i)}_{2})$ is determined by its residue at $P^{(i)}_1$.

Note that not all irreducible components $E_i \subset \X_0$ are inessential. Indeed, this would mean that $\X_0$ is a cycle of rational curves with no marked points, which means that $g = 1$ and $\deg(B|_{X_t}) = 0$, contradicting our assumption that $(X,B)$ is not a family of genus 1 curves with no marked points.

So, without loss of generality, let $E_1 \subset \X_0$ be an essential component. Suppose $E_2$ is an inessential component of $\X_0$ such that $P = E_1 \cap E_2$ is a nodal point in $\X_0$. Let $\psi \in H^0(\omega_{\X_{0}}^{\otimes 1-m}(-\B|_{\X_0}))$. Then, $\psi|_{E_1}$ must be zero. So by compatibility of residues, the residue of $\psi|_{E_2}$ at $P$ must be 0. Thus, $\psi|_{E_2} = 0$. More generally, for any inessential component $E_i$ of $\X_0$, we pick a path joining $v_{E_1}$ and $v_{E_i}$ in $\Gamma_\X$ and apply induction along this path to conclude that $\psi|_{E_i} = 0$. This can be done as $\Gamma_\X$ is connected. Thus, $H^0(\omega_{\X_{0}}^{\otimes 1-m}(-\B|_{\X_0})) = 0$
\end{proof}

Applying Grauert's Lemma \cite[Corollary III.12.9]{Har77} to $\L = \omega_{\X}^{\otimes m}(\B)$, we conclude
\begin{lem}
\label{lemExtendSections}
There exists an open neighborhood $U \subset \X$ of $\X_0$ such that we can find $\theta_1,\dots,\theta_M \in H^0(U,\omega_{\X}^{\otimes m}(\B))$ such that $\theta_i|_{X_t}$ is a basis of $H^0(X_t, \omega_{X_t}^{\otimes m}(B))$ for $|t| \ll 1$ and $\theta_i|_{\X_{0}}$ is a basis for $H^0(\X_{0},\omega_{\X_{0}}^{\otimes m}(\B))$. \qed
\end{lem}


\subsection{A description of global sections}
The following lemma tells us that we can recover the residues of any section $\psi \in H^0(\X_{0},\omega_{\X_{0}}^{\otimes m}(\B|_{\X_0}))$ along an inessential chain by just knowing it on one of the edges in the inessential chain.  
\begin{lem}
\label{lemResidueGenus0Valency2}
Let $v_0,\dots,v_N$ be an inessential chain in  $\Gamma_{\X}$. Let $Q_i$ for $0 \leq i \leq N - 1$ be the nodal point in $\X_0$ corresponding to the edge $v_{i}v_{i+1}$ in the inessential chain. 

Let $\theta \in H^0(\X_{0},\omega_{\X_{0}}^{\otimes m}(B))$. Let $C$ denote the residue of $\theta$ at $Q_0$ along $E_{v_0}$. Then the residue of $\theta$ at $Q_i$ along $E_{v_i}$ is $C$ and the residue at $Q_i$ along $E_{v_{i+1}}$ is $(-1)^{m}C$ for all $0 \leq i \leq N-1$.
\end{lem}
\begin{proof}
  If the residue of $\theta$ at $Q_0$ along $E_{v_0}$ is $C$, then its residue at $Q_{0}$ along $E_{v_1}$ must be $(-1)^mC$ by the compatibility of the residues.
  Note that
  $$\theta|_{E_{v_1}} \in H^0(E_{v_1}, mK_{E_{v_{1}}} + mQ_0 + mQ_1).$$

  We also have that $H^0(E_{v_1}, mK_{E_{v_{1}}} + mQ_0 + mQ_1) \simeq H^0(\P^1,\O_{\P^1})$ is a one-dimensional complex vector space and the map
  $$H^0(E_{v_1}, mK_{E_{v_{1}}} + mQ_0 + mQ_1) \to \C$$
given by taking the residue at $Q_1$ is an isomorphism. The residue of $\theta$ at $Q_0$ and $Q_1$ differ by a factor of $(-1)^m$. Thus, the residue of $\theta$ at $Q_1$ is $C$. Now the proof follows by induction.  
\end{proof}

\begin{lem}
\label{lemSESVSn}
Let $\X$ be the minimal snc model of $(X,\frac{1}{m}B)$.  We have the following short exact sequence of vector spaces.
\begin{multline*}
  0 \to \bigoplus_i H^0(mK_{E_i} + (m-1)P_1^{(i)} + \dots + (m-1)P^{(i)}_{r_i} + \B|_{E_i}) \xrightarrow{\phi} \\
  H^0(\omega_{\X_{0}}^{\otimes m}(\B|_{\X_0}) ) \xrightarrow{\phi'} \C^{E(\tilde\Gamma)} \to 0.
\end{multline*}
\end{lem}
\begin{proof}
  We first describe the maps. The map $\phi$ exists because all the residues of sections in $H^0(mK_{E_i} + (m-1)P_1^{(i)} + \dots + (m-1)P^{(i)}_{r_i}+\B|_{E_i})$ are zero and there is no compatibility of residues that needs to be checked. It is clearly injective since any element of $H^0(\omega_{\X_{0}}^{\otimes m}(\B|_{\X_0}))$ can be recovered from the restrictions to all irreducible components of $\X_{0}$.

  The second map $\phi'$ is defined as follows. Assign an arbitrary orientation to edges in $\tilde\Gamma$. Pick an edge $e \in E(\tilde{\Gamma})$, let $v_0,\dots,v_N$ be the maximal inessential chain associated to the edge $e$, where $v_0$ is the initial vertex. Then $\phi'$ sends an element $\psi \in H^0(\omega_{\X_{0,\red}}^{\otimes n})$ to the residue of $\psi|_{E_{v_0}}$ at the point corresponding to the edge $v_0v_1$.

It is clear that the composition $\phi' \circ \phi$ is $0$ and the exactness at the middle place follows from Lemma \ref{lemResidueGenus0Valency2}. It remains to show that $\phi'$ is surjective, which will follow if we show that all the vector spaces in the above short exact sequence have the right dimensions. 
  
Consider the following long exact sequence induced by the short exact sequence \eqref{eqnSESOmegaN20} tensored with $\O_{\X_0}(\B)$, where we get the last map is surjective by Lemma \ref{lemH1Vanishing}.
\begin{multline}
  0 \to \bigoplus_i H^0(mK_{E_i} + (m-1)P_1^{(i)} + \dots + (m-1)P^{(i)}_{r_i} + \B|_{E_i}) \xrightarrow{\phi} \\
  H^0(\omega_{\X_{0}}^{\otimes m}(\B|_{\X_0})) \to \C^{E(\Gamma_\X)} \to \bigoplus_i H^1(mK_{E_i} + (m-1)P_1^{(i)} + \dots + (m-1)P^{(i)}_{r_i}+\B|_{E_i}) \to 0. 
\end{multline}
By Serre duality,
\begin{multline*}
h^1(mK_{E_i} + (m-1)P_1^{(i)} + \dots + (m-1)P^{(i)}_{r_i} + \B|_{E_i}) = \\  h^0((1-m)K_{E_i} + (1-m)P_1^{(i)} + \dots + (1-m)P^{(i)}_{r_i} - \B|_{E_i}).
\end{multline*}

Following the discussion in the proof of Lemma \ref{lemH1Vanishing}, the above is $0$ unless $E_i$ is inessential, in which case it is 1. Thus, the dimension of the last term in the above long exact sequence is equal to the number of inessential vertices in $\Gamma_\X$. Using $$\#E(\tilde{\Gamma}) = \#E(\Gamma_\X) - \#\{v \in V(\Gamma_\X) \mid v \text{ is inessential}  \},$$
it follows that all the vectors spaces in the short exact sequence in the lemma have the required dimensions. 
\end{proof}

\begin{rmk}
  In the case when $B = 0$, the analog of the above lemma for the case $m = 1$ is \cite[Equation (4.2)]{Shi20A}, which states that we have the following short exact sequence
  $$ 0 \to \bigoplus_i H^0(E_i,K_{E_i}) \to H^0(\X_0,\omega_{\X_0}) \to \Omega(\Gamma_\X) \to 0. $$

Here $\Omega(\Gamma_\X)$ is the collection complex-valued functions on $E(\Gamma_\X)$ which satisfies a balancing condition at all vertices.  The reason for the difference in the two cases is that the global sections of $\omega_{\X_{0}}$ must satisfy the residue theorem at all irreducible components while global sections of $\omega_{\X_{0}}^{\otimes m}$, for some $m \geq 2$, only need to satisfy the residue theorem at irreducible components with genus zero and valency 2. 
\end{rmk}

\subsection{The essential skeleton}
\label{subsecEssentialSkeleton} We show that the dual graph $\Gamma_\X$ of $\X$, the minimal snc model of $(X,\frac{1}{m}B)$ is precisely the essential skeleton of the pair $(X,\frac{1}{m}B)$. We first recall what the essential skeleton is.

Given a pair $(Y,D)$  where $Y$ is a smooth variety over a $\Ct$ and $D$ is a $\Q$-divisor on $Y$, we can obtain a subset of the $Y^\an$ defined by the minimality locus of  certain weight functions \cite{MN15} \cite{BM19}. If all the coefficients of the irreducible components appearing in $D$ are all strictly less than 1, then the essential skeleton of $(Y,D)$ is contained in the dual complex of any snc model of $(Y,D_\red)$. Therefore, to compute the essential skeleton of $(Y,D)$, it is enough to work with any one snc model of $(Y,D_\red)$.  We describe the weight function in our context in the proof of the following lemma.

\begin{lem}
Let $\X$ be the minimal snc model of $(X,\frac{1}{m}B)$. Then, $\Gamma_\X$ is precisely the essential skeleton of the pair $(X_\Ct,\frac{1}{m}\B_\Ct)$. 
\end{lem}
\begin{proof} Note that $\X$ is not an snc model of $(X,B_\red)$. Therefore, we work with the minimal semistable model of $(X,B_\red)$, which we denote as $\X'$. Recall that $\X$ is obtained from  $\X'$ is obtained by repeatedly blowing down those $(-1)$ curves $E$ in the central fiber such that $\deg(\B|_{E}) < m$. Let $p : \X' \to \X$ denote this map. 

Let $\theta \in H^0(\omega_{X/\D^*}^{\otimes km}(k\B))$. Then, we can think of $\theta$ as also being a rational section of $\omega^{\otimes km}_{\X'/\D}(k\X'_{0}+ k\B_1 + \dots + k\B_r)$. Let $\div(\theta)$ denote the associated divisor associated to $\theta$ on $\X'$, when viewed as a rational section of $\omega_{\X'/\D}^{\otimes km}(k\X'_{0}+ k\B_1 + \dots + k\B_r)$. We define a function $\wt_\theta : \Gamma_{\X'} \to \mathbb{\R}$ as
$$ \wt_\theta(x) = \nu_x(\div(\theta)),$$
where $\nu_x$ is the valuation associated to the point $x \in \Gamma_{\X'} \subset X^\an_\Ct$. Here $\nu_x(\div(\theta))$ denotes the valuation $\nu_x$ applied to the equation defining $\div(\theta)$ at the center of the valuation $\nu_x$. 

For example, If $x \in \Gamma_{\X'}$ is a vertex associated to an irreducible component $E \subset \X'_0$, then $\wt_\theta(x)$ is the multiplicity of $E$ in $\div(\theta)$ i.e.~the order of vanishing of $\theta$ along $E$.

Recall that $\Sk(X_\Ct,\frac{1}{m}B_\Ct,\theta) \subset \Gamma_{\X'}$ is the minimalily locus of $\wt_\theta$ i.e.~ $$\Sk(X_\Ct,\frac{1}{m}B_\Ct,\theta) = \{ x \in \Gamma_{\X'} | \wt_\theta(x) = \min_{y \in \Gamma_{\X'}} \wt_\theta(y) \}.$$
The essential skeleton is given by $\Sk(X,\frac{1}{m}B) = \cup_{\theta} \Sk(X,\frac{1}{m}B,\theta)$, where $\theta$ runs over all non-zero elements $H^0(X,\omega^{\otimes km}_{X/\D^*}(k\B))$ for all $k \geq 1$.



Let $\theta_1,\dots,\theta_M$ be the elements of $H^0(\X,\omega^{\otimes km}_{\X/\D}(k\B))$ obtained from Lemma \ref{lemExtendSections}. It is enough to consider those $\theta$ that lie in the linear span of $\theta_1|_X,\dots,\theta_M|_X$ as any section of $H^0(X,\omega^{\otimes km}_{X/\D^*}(k\B))$  would differ from an element in the linear span by a factor of a non-zero element in $\Ct$ (in which case, the weight function would differ by a constant). In this case, the minimum value of $\wt_\theta$ is 0. Let $S$ denote the set of non-zero elements in the linear span of $\theta_1,\dots,\theta_M$ for all choices of $k$. 

If $e$ is an edge in $\Gamma_{\X'}$, then $\wt_{\theta}|_{e} = 0$ iff $p^*(\theta)$ has a pole of order $km$ along the node associated to $P$. If $v$ is a vertex in $\Gamma_{\X'}$, then $\wt_{\theta}(v) = 0$ iff $p^*(\theta)$ does not vanish along the irreducible component associated to $v$.

Thus, it follows from Lemma \ref{lemSESVSn} given an edge $e_P \in E(\Gamma_\X)$, there exists a $\theta$ such that $\theta$ has a pole of order $m$ along $P$. Similarly, given a vertex $v_E \in V(\Gamma_\X)$ there exists a $\theta$ such that $\theta$ does not vanish along $E$. Thus, $\Gamma_\X$ is contained in the essential skeleton.

To show that $\Gamma_\X$ contains the essential skeleton, recall that $\X$ is obtained from  $\X'$ is obtained by repeatedly blowing down those $(-1)$ curves $E$ in the central fiber such that $\deg(\B|_{E}) < m$.

Consider a $(-1)$-curve $E \subset \X_0'$. Then, there is only one nodal point of $\X_0'$ that is contained in $E$. Let $P$ denote this nodal point of $\X_0'$ contained in $E$. 

Since $\X_0'$ is a principal divisor, we have that $\omega_{\X'_0} \simeq \omega_{\X_0'}(\X_0')$. By the adjunction formula, we also have that have that $$\omega_{\X_0'}|_{E} \simeq \omega_{\X_0'}(\X'_0)|_E \simeq \omega_{\X_0}(E)|_E \otimes \O_{E}(\overline{(\X'_0 \setminus E)} \cap E) \simeq \omega_{E} \otimes \O_E(P)$$
Therefore, 
\begin{align*}
  \omega_{\X'_0}^{\otimes km}(k\B)|_{E}  &\simeq \omega_{E}^{\otimes km}(kmP + k\B|_{E})                                       
  \\ &\simeq \O_{\P^1}(-km + k\deg(\B|_{E})).
\end{align*}

Since $\deg(\B|_{E}) < m$, $\O_{\P^1}(-km + k\deg(\B|_{E}))$ has no global sections. Thus, $\theta|_{E} = 0$ and $\theta|_{\X'_0}$ does not have a pole of order $m$ at $P$ for all $\theta \in S$. Thus, $e_P$ and $v_E$ do not lie in the essential skeleton.

More generally, given an edge $e_P$ not lying the essential skeleton, we can factor $\X' \xrightarrow{q'} \X'' \xrightarrow{q''} \X$ such that $p = q'' \circ q'$ and $q'$, $q''$ are a series of blow downs such that $P = E \cap E_1$ in $\X''$, where $E$ is a $(-1)$-curve in $\X''$ with $\deg(\B|_{E}) < m$ in $\X''$. Repeating the previous argument, we get that $(q'')^*\theta$, and thus $p^*\theta$, vanish on $E$ and do not have a pole of order $m$ along $P$. Thus, $e_P$ and $v_{E}$ do not lie in the essential skeleton. 
\end{proof}

\section{Convergence of the Narasimhan Simha measure}
\label{secConvergenceNS}
We study the convergence of the Narasimhan--Simha measure in this section.

\subsection{Setup and notation}
Let $X \to \D^*$ be a holomorphic family of genus $g$ curves. Let $B = b_1B_1 + \dots + b_rB_r$ be a horizontal divisor in $X$. Let $m$ be an integer such that $b_i < m$ for all $i=1,\dots,r$.

Let $\X$ be the minimal snc model of $(X,\frac{1}{m}B)$ and let $m \geq 2$ be a fixed integer. Let $M = (2m-1)(g-1) + \deg(B|_{X_t})$ and let $s = \# E(\tilde{\Gamma})$. Let $\tau_t$ denote the Narasimhan--Simha volume form on $X_t$ with respect to the line bundle $\Omega_{X_t}^{\otimes m}(B|_{X_t})$ and let $\mu_t$ denote the pluri-Bergman measure on $X_t$ with respect to the line bundle $\Omega_{X_t}^{\otimes m}(B|_{X_t})$. 

Enumerate the edges of $E(\tilde\Gamma)$ as $e_1,\dots,e_s$. Since all edges of $\Gamma_\X$ have length $1$, if $e_i \in E(\tilde\Gamma)$ corresponds to the maximal inessential chain $v_0,v_1,\dots,v_N$ in $\Gamma_\X$, then $l_{e_i} = N$. (See Section \ref{subsecDualGraph} for details.)

Using Lemma \ref{lemSESVSn}, we can pick a basis $\psi_1,\dots,\psi_M$ of $H^0(\X_0,\omega_{\X_{0}}^{\otimes m}(\B|_{\X_0}))$ such that $\psi_1,\dots,\psi_s$ map to the standard basis of $\C^{E(\tilde \Gamma)} = \C^s$ and $\psi_{s+1},\dots,\psi_M$ give rise to an orthonormal basis of $\bigoplus_{i=1}^m H^0(E_i,mK_{E_i} + (m-1)P^{(i)}_1 + \dots + (m-1)P^{(i)}_{r_i} + \B|_{E_i})$ with respect to the Hermitian pairing \eqref{eqnGenBergmanPairing} on each summand. In particular, for $1 \leq i \leq s$, $\psi_i$ has residues of $\pm 1$ at those nodal points of $\X_0$ that lie on the maximal inessential chain associated to  $e_i$ and has zero residues at all other points. For $s+1 \leq i \leq M$, $\psi_i$ has zero residues at all the nodal points of $\X_0$.   



We say that $E_1$ is a Type I component if $h^0(E_i,mK_{E_i} + (m-1)P^{(i)}_1 + \dots + (m-1)P^{(i)}_{r_i} + \B|_{E_i}) > 0$, otherwise it is called a Type II component. It is easy to check that there are only the following possible choices for a Type II component $E$.
\begin{itemize}
\item $m=2$, $g(E) = 0$, $\val(E) = 3$ and $\deg(\B|_E) = 0$
\item $g(E) = 0$, $\val(E) = 2$ and $\deg(\B|_{E}) = 1$.
\item $E$ is inessential i.e. $g(E) = 0$, $\val(E) = 2$ and $\deg(\B|_{E}) = 0$.
\item $g(E) = 0$, $\val(E) = 1$ and $\deg(\B|_E) = m$
\end{itemize}

The Type II components will precisely be the curves in $\Delta_\CC(\X)$ on which the limiting measure $\tau_0$ and $\mu_0$ place no mass.

Let $\tilde{\tau_0}$ denote the Narasimhan--Simha volume form on $\tilde{\X_0}$ with respect to \\
$\bigoplus_{i=1}^m H^0(E_i,mK_{E_i} + (m-1)P^{(i)}_1 + \dots + (m-1)P^{(i)}_{r_i} + \B|_{E_i})$ i.e.~if $E_i$ is a Type I component, then $\tilde{\tau}_0|_{E_i}$ is the Narasimhan--Simha volume form on $E_i$ with respect to $mK_{E_i} + (m-1)P^{(i)}_1 + \dots + (m-1)P^{(i)}_{r_i} + \B|_{E_i}$.  If $E_i$ is a Type II component, then $ h^0(E_i,mK_{E_i} + (m-1)P^{(i)}_1 + \dots + (m-1)P^{(i)}_{r_i} + \B|_{E_i}) = 0$  and we just set $\tilde{\tau}_0|_{E_i} = 0$.

Note that if $E$ is of Type II, then $\psi_{s+1}|_{E},\dots,\psi_M|_{E}  = 0$ as $\psi_{s+1},\dots, \psi_{M}$ form a basis of $\bigoplus_{i=1}^m H^0(E_i,mK_{E_i} + (m-1)P^{(i)}_1 + \dots + (m-1)P^{(i)}_{r_i} + \B|_{E_i})$ and $h^0(E,mK_{E} + (m-1)P_1 + \dots + (m-1)P_{r_i} + \B|_{E}) = 0$. Similarly, if $E$ is of Type I, $\psi_{s+1},\dots,\psi_M$ do not simultaneously vanish at any point of $E$.

The following Theorem is a more precise version of Theorem \ref{mainThmA} for the minimal snc model, $\X$, of $(X,\frac{1}{m}B)$. 
\begin{thm}
\label{thmA}
  Let $\tau_0$ denote the measure on $\Delta_\CC(\X)$ given by $\tilde{\tau}_0$ on curves and by taking the Lebesgue measure on an edge $e \in E(\Gamma_\X)$ of length $\frac{1}{l_e}$, where $l_e$ is the length of the maximal inessential chain in $\Gamma_\X$ containing the edge $e$.
  
The measures $\tau_t$ converge to the measure $\tau_0$ when viewed as measures on $\X^\hyb_\CC$. 
\end{thm}
\begin{proof}
Theorem \ref{thmA} follows directly from Lemma \ref{lemConvergenceMCCHC} and Corollaries \ref{corA1} -- \ref{corA3}.
\end{proof}

Using Lemma \ref{lemExtendSections}, we pick $\theta_1,\dots,\theta_M \in H^0(U,\omega_{\X}^{\otimes m}(\B))$ for a neighborhood $U \subset \X$ of $\X_0$ such that $\theta_1|_{X_t},\dots,\theta_M|_{X_t}$ form a basis of $H^0(X_t,\omega^{\otimes m}_{X_t}(B|_{X_t}))$ and $\theta_1|_{\X_0},\dots,\theta_M|_{\X_0}$ form a basis of $H^0(\X_0,\omega^{\otimes m}_{\X_0}(\B|_{\X_0}))$.

After applying a $\C$-linear transformation to $\theta_1,\dots,\theta_M$, we may assume that $\theta_j|_{\X_{0}} = \psi_i$ for $1 \leq i \leq M$. Denote $\theta_{j,t} = \theta_j|_{X_t}$.

Let $(U_1,t,w)$ be a coordinate chart in $\X$ adapted to an irreducible component $E \subset \X_0$. Recall that the coordinates in $U_1$ are $t,w$ with $|t|,|w| < 1$ and the projection $U_1 \to \D$ is given by $(t,w) \mapsto t$. We may shrink $U_1$ to suppose that either $U_1 \cap E \cap \B = \emptyset$ or $U_1 \cap E \cap \B = \{ (0,0)\}$ and that $\theta_j$'s admit a power series expansion:
$$ \theta_j(t,w) = \sum_{\alpha,\beta \geq 0} \frac{c^{(j)}_{\alpha,\beta} t^\alpha w^\beta}{\phi_j(t,w)} (dw \wedge dt)^{\otimes m},$$
where $\phi_j(t,w) = 0$ is a local equation of $\B$ in $U$. 
Then,
$$\theta_{j,t}(w) = \sum_{\alpha,\beta \geq 0} \frac{c^{(j)}_{\alpha,\beta} t^\alpha w^\beta}{\phi_j(t,w)} dw^{\otimes m} $$
and
$$ \psi_{j}(w) = \sum_{\beta \geq 0} \frac{c^{(j)}_{0,\beta}w^{\beta}}{\phi_{j}(0,w)} dw^{\otimes m}.$$

Note that since $U \cap E \cap \B = \emptyset $ or $U \cap E \cap \B =  \{(0,0) \}$, we may pick $\phi_j$ so that $\phi_j(0,w) = w^{k}$ for some integer $0 \leq k < m$. Thus, we get that 
\begin{equation}
  \label{eqnThetaJTR1} 
\theta_{j,t}(w) = \psi_j(w) + O(|w|^{1-m}|t|)
\end{equation}
as $t \to 0$ for fixed $w \in \D$ and for all $1 \leq j \leq M$. Moreover,  $|\theta_{j,t}(w)|^{2/m}$  are bounded by an integrable function (for example, $\frac{C}{|w|^{2k/m}}$) for $t$ small enough and we are in the setting to apply the dominated convergence theorem. Throughout this paper, most pointwise convergences that show up will be in the setting to apply the dominated convergence theorem and in most cases, we do not mention explicitly mention a dominating integrable function as it would be easy to find one.

Now let $(U_2,z,w)$ be a coordinate chart in $\X$ adapted to a node $P = E_1 \cap E_2 \in \X_0$ such that $U_2 \cap \B = \emptyset$. Recall that the coordinates in $U_2$ are $z,w$ such that $|z|,|w| < 1$, $E_1 = \{z = 0 \}$, $E_2 = \{w = 0\}$, and the projection $U_2 \to \D$  is given by $(z,w) \mapsto zw$.  We may shrink $U_2$ to suppose that $\theta_j$'s admit a power series expansion.  
$$ \theta_j(t,w) = \sum_{\alpha,\beta \geq 0} c^{(j)}_{\alpha,\beta} z^\alpha w^\beta (dw \wedge dz)^{\otimes m}.$$

Then, for $ |t| < |w| < 1$, 
\begin{equation}
\label{eqnPluriThetaJTPowerSeries}
\theta_{j,t}(w) = \sum_{\alpha,\beta \geq 0} c^{(j)}_{\alpha,\beta} t^{\alpha} w^{\beta-\alpha-m} dw^{\otimes m}.
\end{equation} 
We also have that
$$ \psi_{j}(w) = \sum_{\alpha\geq 0} c_{0,\beta}^{(i)}w^{\beta-m}dw^{\otimes m}.$$
on $U_2 \cap E_1$, where we think of $w$ as a coordinate on $E_1 \cap U_2$.
Thus, we see that $c^{(j)}_{0,0}$ is the residue of $\psi_j$ at $P$. Thus, $c^{(j)}_{0,0}$ is $\pm 1$ if $\psi_j$ has a  pole of order $m$ at $P$, otherwise it is $0$.

Consider the regions 
\begin{align*}
  R_1 &= \left\{ (z,w) \in U_2 \ \Big| \ |t|^{1/2} < |w| < \frac{1}{(\log|t|^{-1})^{m}}  \right\} \\
  R_2 &= \left\{ (z,w) \in U_2 \ \Big| \ \frac{1}{(\log|t|^{-1})^{m}} < |w| < 1 \right\}.
\end{align*}

Let us figure out the dominating terms of $\theta_{j,t}$ in each of these regions. Without loss of generality, suppose that $\psi_1$ develops a pole of order $m$ at $P$ with residue $1$. Then, the $\psi_2,\dots,\psi_M$ can have poles of order at worst $m-1$ at $P$. 
From equation, \eqref{eqnPluriThetaJTPowerSeries}, we get that
$$ \theta_{1,t}(w) =  w^{-m}\left(1 + \sum_{(\alpha,\beta) \neq (0,0)} c_{\alpha,\beta}^{(j)}w^{\beta-\alpha}t^\alpha \right) dw^{\otimes m}$$
After shrinking $U_2$ and rescaling $z,w,t$, we may assume that $\sum_{\alpha,\beta}|c_{\alpha,\beta}| < \infty$. 
In the region $R_1$, we have that $|t|^{1/2} < |w| < \frac{1}{(\log|t|^{-1})^m}$. Thus,
\begin{align*}
  |w^{\beta - \alpha}t^\alpha| &\leq |t|^{\frac{\beta - \alpha}{2}}|t|^\alpha = |t|^{\frac{\beta+\alpha}{2}}  \text{ if } \alpha \geq \beta\\
  |w^{\beta - \alpha}t^\alpha| &\leq \left(\frac{1}{\log|t|^{-1}}\right)^{m(\beta - \alpha)}|t|^\alpha \text{ if }  \beta \geq \alpha,
\end{align*} 
and we get that
$$ \left|\sum_{\alpha,\beta \neq (0,0)} c_{\alpha,\beta}w^{\beta - \alpha}t^\alpha \right| = O\left( \frac{1}{(\log|t|^{-1})^m} \right)$$
and
\begin{equation}
\label{eqnTheta1TR1}
 \theta_{1,t} \approx w^{-m} dw^{\otimes m}\text{ in } R_1.
\end{equation}
Similarly, for $2 \leq j \leq M$, we get that
\begin{equation}
 \label{eqnThetaJTR1}
\theta_{j,t} \approx c^{(j)}_{0,m-m_j}w^{-m_j} dw^{\otimes m} \text{ in } R_1,
\end{equation}
where $m_j < m$ is order of the pole of $\psi_j$ at $P$.


In the region $R_2$, we can write $\theta_{j,t} = \psi_j + \sum_{\alpha \geq 1, \beta}c_{\alpha,\beta}^{(j)}w^{\beta-\alpha-m}t^\alpha$ and we see that
\begin{equation}
\label{eqnThetaJTR2}
\theta_{j,t}(w) \to \psi_{j}(w) 
\end{equation}
as $t \to 0$ for a fixed $w \in \D^*$ for all $1 \leq j \leq M$.

\subsection{Asymptotics of $\|\theta_{j,t}\|'_{X_t}$}
Recall from Section \ref{subsecPluriBergman} that for a Riemann surface $Y$ and a meromorphic $m$-canonical form $\vartheta$ on $Y$, we denote
$$ \| \vartheta \|_{Y}' := \left( \int_{Y} |\vartheta|^{2/m} \right)^{m/2}.$$
Note that the above also makes sense if $Y$ is a disconnected Riemann surface. One of the key things in the definition of $\tau_t$ is the condition that $\| \cdot \|'_{X_t} = 1$. Therefore to understand the asymptotics of $\tau_t$, we first need to understand $\| \cdot \|'_{X_t}$. We begin by looking at $\|\theta_{j,t}\|'_{X_t}$. 

\begin{lem}
\label{lemAsymptoticsPseudonorm}
  For $1 \leq j \leq s$, $$\| \theta_{j,t} \|'_{X_t} \approx (2\pi l_{e_j}\log|t|^{-1})^{m/2} $$

  and for $s+1 \leq j \leq M$,
  $$\| \theta_{j,t} \|'_{X_t} \to \| \psi_{j} \|'_{\tilde{\X_0}}$$
 as $t \to 0$.
\end{lem}
\begin{proof}
  By using a partition of unity argument, we may reduce the problem to finding the asymptotics on adapted coordinate chats.

  If $(U_1,t,w)$ is a coordinate chart adapted to an irreducible component  $E \subset \X_0$, then using Equation \eqref{eqnThetaJTR2}, we get that $\theta_{j,t} \to \psi_j$ as $t \to 0$. Using the dominated convergence theorem, we get that
  $$\int_{X_t \cap U_1} |\theta_{j,t}|^{2/m} \to \int_{E \cap U_1} |\psi_j|^{2/m}$$
  for all $1 \leq j \leq M$.

If $U_2$ is a coordinate chart adapted to a node $P = E_1 \cap E_2$. First consider the case when $1 \leq j \leq s$, then it follows from Equation \eqref{eqnPluriThetaJTPowerSeries} that on the set
$\{ |t|^{1/2} < |w| < 1 \}$,
$$ \theta_{j,t} = \left( \frac{C_j}{w^m} + O(|w|^{1-m}) \right) dw^{\otimes m},$$
where the $O(|w|^{1-m})$ is with respect to $w$ as $w \to 0$ and is uniform in $t$, and $C_j = \pm 1$ if $\psi_j$ has a pole of order $m$ at $P$, otherwise $C_j = 0$.

Thus in the region $\{ |t|^{1/2} < |w| < 1 \}$, $$|\theta_{j,t}|^{2/m} = \left( \frac{|C_j|}{|w|^2} + O\left(\frac{1}{|w|^{\frac{2(m-1)}{m}}}\right) \right) |dw \wedge d\overline{w}|.$$

Thus, we get that $\int_{|t|^{1/2} < |w| < 1} |\theta_{j,t}|^{2/m} = |C_j|\pi \log{|t|^{-1} + O(1)}$ as $t \to 0$.

Similarly, we get that $\int_{|t|^{1/2} < |z| < 1} |\theta_{j,t}|^{2/m} = |C_j|\pi \log{|t|^{-1} + O(1)}$ and thus
$$\int_{U_2 \cap X_t} |\theta_{j,t}|^{2/m} = 2|C_j| \pi\log|t|^{-1} + O(1).$$

Since, $\psi_j$ has a pole of order $m$ at $l_{e_j}$ many points, we get that
$$ \int_{X_t} |\theta_{j,t}|^{2/m} = 2\pi l_{e_j} \log|t|^{-1} + O(1),$$
and thus we get the required estimate for $1 \leq j \leq s$.

In the case when $s+1 \leq j \leq M$, then on the set $\{ |t|^{1/2} < |w| < 1 \}$, we have that $|\theta_{j,t}|^{2/m} \to |\psi_j|^{2/m}$ as $t \to 0$. Furthermore, we have using Equation \eqref{eqnPluriThetaJTPowerSeries} that
$$ \theta_{j,t}(w) - \psi_j(w) = \sum_{\alpha \geq 1, \beta \geq 0} c_{\alpha, \beta}^{(j)} t^\alpha w^{\beta-\alpha-m} dw^{\otimes m}.$$

On the region $\{ |t|^{1/2} < |w| < 1 \}$, we get $|t^\alpha w^{\beta-\alpha-m}| < |w|^{\beta + \alpha -m}$ and we get that the right hand side in above equation is uniformly bounded by $C|w|^{1-m} dw^{\otimes m}$. Since $|\psi_j(w)|^{2/m}$ and $|w|^{2(1-m)/m}$ are integrable on $\D$ for $s+1\leq j \leq M$, by the dominated convergence theorem, we get that
$$ \int_{|t|^{1/2} < |w| < 1} |\theta_{j,t}|^{2/m} \to \int_{U_2 \cap E_1} |\psi_j|^{2/m},$$
and we get
$$ \int_{X_t} |\theta_{j,t}|^{2/m} \to \int_{\tilde{\X_0}} |\psi_j|^{2/m}.$$
\end{proof}

To understand the asymptotics of $\| \cdot \|'_{X_t}$, let us denote
$$ \tilde{\theta}_{i,t} := \frac{\theta_{i,t}}{(2\pi l_{e_i}\log|t|^{-1})^{m/2} } 
\text{ for } 1 \leq i \leq s,$$
$$ \tilde{\theta}_{i,t} := \frac{\theta_{i,t}}{\| \psi_i \|'_{\tilde{\X_0}} } \text{ and } \tilde{\psi}_{i} := \frac{\psi_i}{\| \psi_i \|'_{\tilde{\X_0}}}  \text{ for } s+1 \leq i \leq M.$$

The previous lemma tells us that $\|\ttheta_{i,t} \|'_{X_t} \to 1$ as $t \to 0$. Moreover, the following lemma tells us that $\ttheta_{i,t}$ are a `nice' basis of $\C^M$ with respect to $\| \cdot \|'_{X_t}$ as $t \to 0$.

\begin{lem}
\label{lemNormEstimate}
  Let $c_1,\dots,c_M \in \C$. Then,
\begin{multline*}
 \| c_1 \ttheta_{1,t} + \dots c_M \ttheta_{M,t} \|'_{X_t} \approx \left(\sum_{i=1}^s |c_i|^{2/m} + (\| c_{s+1}\tpsi_{s+1} + \dots  c_M\tpsi_{M}\|_{\tilde{\X_0}}')^{2/m}\right)^{m/2} 
\end{multline*}
  
  as $t \to 0$.
\end{lem}
\begin{proof}
  \newcommand{\cc}{c_1\ttheta_{1,t} + \dots+ c_M\ttheta_{M,t}}
  We use a partition of unity argument to reduce the problem to adapted coordinate charts. If $(U_1,w)$ is an coordinate chart adapted to an irreducible component of $\X_0$, then for $1 \leq j \leq s$,


\begin{equation}
  \label{eqnPf771Eqn0}
|\cc|^{2/m}(w) \to |c_{s+1}\tpsi_{s+1}+\dots+c_M\tpsi_M|^{2/m}(w)
\end{equation}
pointwise for a fixed $w$. Moreover, there exists an integrable function on $\D$ which dominates $|\cc|^{2/m}$ for all $t$ small enough. To see this, just note that there exists a constant $C$ such that
$$ |\cc|^{2/m}  \leq C \max_{j,k}|\ttheta_{j,t} \wedge \overline{\ttheta_{k,t}}|^{1/m}$$
and each of the $|\ttheta_{j,t} \wedge \overline{\ttheta_{k,t}}|^{1/m}$ are themselves bounded by an integrable function on $\D$. Thus, by the dominated convergence theorem, we get that
\begin{equation}
  \label{eqnPf771Eqn1}
\int_{U_1 \cap X_t} |\cc|^{2/m}(w) \to \int_{U_1 \cap E} |c_{s+1}\tpsi_{s+1}+\dots+c_M\tpsi_M|^{2/m}(w)
\end{equation}

Now consider a coordinate chart $(U_2,z,w)$ adapted to a node $P = E_1 \cap E_2$. Without loss of generality, suppose that $\psi_1$ develops a pole of order $m$ at $P$. To analyze the integral $\int_{|t|^{1/2} < |w| < 1}|\cc|^{2/m}$, we break up $\{ |t|^{1/2} < |w| < 1 \}$ into two regions: $$R_1 = \left\{ |t|^{1/2} < |w| < \frac{1}{(\log|t|^{-1})^m}\right\}$$ and
$$R_2 = \left\{ \frac{1}{(\log|t|^{-1})^m} < |w| < 1  \right\}.$$

On the region $R_1$, using Equations \eqref{eqnTheta1TR1} and \eqref{eqnThetaJTR1}, 
$$\left| \frac{\ttheta_{j,t}}{\ttheta_{1,t}} \right| \leq C (\log|t|^{-1})^{m/2}|w| = O\left(\frac{1}{(\log|t|^{-1})^{m/2}}\right),$$
for $2 \leq j \leq M$, where the second equality follows from the fact that $|w| < \frac{1}{\log(|t|^{-1})^{m}}$ in this region.

Thus, in this region,
\begin{align}
\label{eqnPf771Eqn4}
|\cc|^{2/m} \approx |c_1|^{2/m}|\ttheta_{1,t}|^{2/m}  \approx  \frac{|c_1|^{2/m}dw \wedge d\overline{w}}{2 \pi l_{e_1} |w|^2 \log|t|^{-1}},
\end{align}
where the second equality follows from Equation \eqref{eqnTheta1TR1}.

It is easy to verify that $$\int_{|t|^{1/2} < |w| < 1} \frac{dw \wedge d\overline{w}}{2 \pi l_{e_1} |w|^2 \log|t|^{-1}} \to \frac{1}{2l_{e_1}} .$$ Thus, we get that
\begin{equation}
 \label{eqnPf771Eqn5}
\int_{R_1 \cap X_t} |\cc|^{2/m} \to \frac{|c_1|^{2/m}}{2l_{e_1}}
\end{equation}
as $t \to 0$.


To analyze $\int_{R_2 \cap X_t} |\cc|^{2/m}$,
note that $|\cc|^{2/m}(w) \to |c_{s+1}\tpsi_{s+1}+ \dots + c_M\tpsi_M|^{2/m}(w)$ as $t \to 0$ for a fixed $w \in \D^*$ and thus,




\begin{equation}
\label{eqnPf771Eqn9}
\int_{R_2 \cap U} |\cc|^{2/m} \to \int_{E_1 \cap U}|c_{s+1}\tpsi_{s+1}+ \dots + c_M\tpsi_M|^{2/m}
\end{equation}

as $t \to 0$.


Combining Equations \eqref{eqnPf771Eqn1},  \eqref{eqnPf771Eqn5}, and \eqref{eqnPf771Eqn9}, we get that
\begin{multline*}
\int_{X_t} |\cc|^{2/m} \to \sum_{i=1}^s|c_i|^{2/m} + \int_{\tilde{\X_0}} |c_{s+1}\tpsi_{s+1} + \dots + c_M\tpsi_M |^{2/m} 
\end{multline*} 
as $t \to 0$ and the result follows.
\end{proof}

\subsection{Asymptotics of $\tau_t$}

\begin{cor}
  \label{corAsymptoticsTau}
  Let $(U_1,t,w)$ be a coordinate chart adapted to a an irreducible component $E \subset \X_0$. Then 
  $$ \tau_t \to \tilde{\tau}_0 $$
  as $t \to 0$. 
  Let $(U_2,z,w)$ be a coordinate chart adapted to a node $P = E_1 \cap E_2$. Then, in the region $\{ |t|^{1/2} < |w| < \frac{1}{(\log |t|^{-1})^m} \}$,
  $$ \tau_t \approx \frac{|dw \wedge d\overline{w}|}{2 \pi l_{e_P} |w|^2\log|t|^{-1}},$$
where $l$ is the length of the edge in $\tilde{\Gamma}$ containing $e_P$ and in the region $\{ \frac{1}{(\log|t|^{-1})^{m}} <  |w| < 1\}$,
  $$ \tau_t \to \tilde{\tau}_0 .$$
\end{cor}
\begin{proof}
  \newcommand{\cc}{c_1\ttheta_{1,t} + \dots+ c_M\ttheta_{M,t}}
  It follows from Lemma \ref{lemNormEstimate} that
 \begin{multline}
\tau_t \approx \max_{|c_1|^{2/m} + \dots + |c_s|^{2/m} + \|c_{s+1}\tpsi_{s+1} + \dots c_M\tpsi_M \|_{\tilde{\X_0}}^{2/m} = 1} |\cc|^{2/m}. 
\end{multline}
Consider the coordinate chart $(U_1,t,w)$ adapted to $E \subset \X_0$. It follows from Equation \eqref{eqnPf771Eqn0} that for a fixed $w \in \D^*$, 
$$ |\cc|^{2/m}(w) \to |c_{s+1}\tpsi_{s+1}+\dots + c_{M}\tpsi_M|^{2/m}(w).$$
Therefore to maximize, $|\cc|^{2/m}(w)$, we need to pick $c_1=\dots=c_{s}=0$ and we get
\begin{align*}
\tau_t(w) \approx \max_{\|c_{s+1}\tpsi_{s+1} + \dots c_M\tpsi_M \|_{\tilde{\X_0}}^{2/m} = 1} |c_{s+1}\tpsi_{s+1}+\dots+c_M\tpsi_M|^{2/m}(w) = \tilde{\tau}_0(w).
\end{align*}

Thus, $\tau_t(w) \to \tilde{\tau}_0(w)$ pointwise for $w \in \D^*$. 

Similarly, the other assertion follows from Equations \eqref{eqnPf771Eqn4}.
\end{proof}

\begin{cor}
\label{corTauTrivialBound}
  There exists a constant $C$ such that for $|t|$ small enough such that
  $$C^{-1} \max_{j,k} \left\{\frac{|\theta_{j,t}\wedge \overline{\theta_{k,t}}|^{1/m}}{(\log|t|^{-1})^{\eta_{j} + \eta_k}}\right\} \leq \tau_t \leq C \max_{j,k} \left\{\frac{|\theta_{j,t}\wedge \overline{\theta_{k,t}}|^{1/m}}{(\log|t|^{-1})^{\eta_{j} + \eta_k}}\right\}$$
  where $\eta_{j} = \frac{1}{2}$ if $1 \leq j \leq s$ and $\eta_{j} = 0$ if $s+1 \leq j \leq M$.
\end{cor}
\begin{proof}
  It is enough to show that there a constant $C$ such that for $|t|$ small enough 
  $$C^{-1} \max_{j,k} \left\{|\ttheta_{j,t}\wedge \overline{\ttheta_{k,t}}|^{1/m}\right\} \leq \tau_t \leq C \max_{j,k} \left\{|\ttheta_{j,t}\wedge \overline{\ttheta_{k,t}}|^{1/m}\right\}$$
  \newcommand{\cc}{c_1\ttheta_{1,t} + \dots+ c_M\ttheta_{M,t}}
  It follows from Lemma \ref{lemNormEstimate} that
 \begin{multline}
\tau_t \approx \max_{|c_1|^{2/m} + \dots + |c_s|^{2/m} + \|c_{s+1}\tpsi_{s+1} + \dots c_M\tpsi_M \|_{\tilde{\X_0}}^{2/m} = 1} |\cc|^{2/m}. 
\end{multline}

Thus, $$\tau_t < 2 \max_{|c_1|^{2/m} + \dots + |c_s|^{2/m} + \|c_{s+1}\tpsi_{s+1} + \dots c_M\tpsi_M \|_{\tilde{\X_0}}^{2/m} = 1} |\cc|^{2/m}$$ for $t$ small enough. 
The constraint $$|c_1|^{2/m} + \dots + |c_s|^{2/m} + \|c_{s+1}\tpsi_{s+1} + \dots c_M\tpsi_M \|_{\tilde{\X_0}}^{2/m} = 1 $$ ensures that there exist a constant $C_1$ such that $|c_1|,\dots,|c_M|  \leq C_1$. Thus, we get a constant $C_2$ such that
$$ \max_{|c_1|^{2/m} + \dots + |c_s|^{2/m} + \|c_{s+1}\tpsi_{s+1} + \dots c_M\tpsi_M \|_{\tilde{\X_0}}^{2/m} = 1} |\cc|^{2/m} \leq C_2 \max_{j,k}|\ttheta_{j,t} \wedge \overline{\ttheta}_{k,t}|^{1/m},$$
which gives us one of the inequalities.

For the other inequality, note that we also have
$$\tau_t > \frac{1}{2} \max_{|c_1|^{2/m} + \dots + |c_s|^{2/m} + \|c_{s+1}\tpsi_{s+1} + \dots c_M\tpsi_M \|_{\tilde{\X_0}}^{2/m} = 1} |\cc|^{2/m}$$ for $t$ small enough.
Setting $c_j = 1$ and the rest 0, we get that $\tau_t \geq \frac{1}{2}|\ttheta_{j,t}|^{2/m}$ for all $1 \leq j,k \leq M$. Since $|\ttheta_{j,t} \wedge \overline{\ttheta}_{k,t}|^{1/m}$ is the geometric mean of $|\ttheta_{j,t}|^{2/m}$ and $|\ttheta_{k,t}|^{2/m}$, we also get that
$$ \tau_{t} \geq \frac{1}{2} |\ttheta_{j,t} \wedge \overline{\ttheta}_{k,t}|^{1/m}$$ for all $1 \leq j,k \leq M$ which gives us the second inequality. 
\end{proof}

The following three corollaries along with Lemma \ref{lemConvergenceMCCHC} prove Theorem \ref{thmA}.
\begin{cor}
\label{corA1}
Let $(U_1,t,w)$ be a coordinate chart adapted to an irreducible component $E$ of $\X_0$. Let $f$ be a compactly supported continuous function on $U_1$. Then, $\int_{X_t \cap U_1}\chi \tau_t \to \int_{E \cap U_1} \chi \tilde{\tau_0}$.
\end{cor}
\begin{proof}
Using Corollary   \ref{corAsymptoticsTau}, we have that $\tau_t \to \tilde{\tau}_0$ in $U_1$. Thus, we get that
$$ \int_{X_t \cap U_1} f \tau_t \to \int_{E \cap U_1} f \tilde{\tau}_0 .$$
\end{proof}

\begin{cor}
 \label{corA2}
  Let $(U_2,z,w)$ be a coordinate chart adapted to a node $P = E_1 \cap E_2$ of $\X_0$. Suppose that $e_1$ is the edge in $\tilde{\Gamma}$ that contains $e_P$. Let $f$ be a continuous function on $[0,1]$ and let $0 < \alpha < \beta < \frac{1}{2}$. Then,
  $$ \int_{|t|^\beta < |w| < |t|^\alpha} f\left(\frac{\log|w|}{\log|t|}\right) \tau_t\to \frac{1}{l_{e_1}} \int_{\alpha}^\beta f(u)du$$
as $t \to 0$.
\end{cor}
\begin{proof}

  Since $e_1$ is the edge in $\tilde{\Gamma}$ containing $e_P$,  $\psi_1$ develops a pole of order $m$ along $P$. Then, using Corollary \ref{corAsymptoticsTau}, we have that $\tau_t \approx  \frac{|dw \wedge d\overline{w}|}{2\pi|w|^2 l_{e_1} \log|t|^{-1}}$. Thus, we are interested in computing the limit
  $$  \int_{|t|^\beta < |w| < |t|^\alpha} f\left(\frac{\log|w|}{\log|t|}\right) \frac{|dw \wedge d\overline{w}|}{2\pi|w|^2 l_{e_1} \log|t|^{-1}}.$$
  Using a change of variables $u = \frac{\log|w|}{\log|t|}$ and $\vartheta = \arg(w)$, we get that the above limit of the above expression as $t \to 0$ is
  $\frac{1}{l_{e_1}} \int_{\alpha}^\beta f(u)du  $.
\end{proof}

\begin{cor}
 \label{corA3}
  Let $(U_2,z,w)$ be a coordinate chart adapted to a node $P = E_1 \cap E_2$ of $\X_0$. Suppose that $e_1$ is the edge in $E(\tilde{\Gamma})$ that contains $e_P$. Let $f$ be a continuous function on $[0,1]$, let $0 < \epsilon \ll \frac{1}{2}$, let $f$ be a compactly supported function on the half-dumbbell $D = \{(w,u) \in \D \times [0,\epsilon) \mid \text{ either } w = 0 \text{ or }  u = 0 \}$ and let $r : \D \times [0,\epsilon) \to D$ be a strong deformation retract. Then,
  $$ \int_{|t|^\epsilon < |w| < 1} f\left(r\left(w,\frac{\log|w|}{\log|t|}\right)\right) \tau_t \to \frac{1}{l_{e_1}}  \int_{0}^\epsilon f(u)du + \int_{E_1 \cap U} f(w)  \tilde{\tau}_0$$
as $t \to 0$.
\end{cor}
\begin{proof}
  To analyze the limit of the integral, we break up the region $\{ |t|^\epsilon < |w|  < 1 \}$ into two parts:  $\{|t|^{\epsilon} < |w| < \frac{1}{(\log|t|^{-1})^m} \}$,  and $\{ \frac{1}{(\log|t|^{-1})^{m}} < |w| < 1 \}$.

In the region, $\{ |t|^{\epsilon} < |w| < \frac{1}{(\log|t|^{-1})^m} \}$, note that $\tau_t \approx \frac{|dw \wedge d\overline{w}|}{2\pi |w|^2 \log|t|^{-1}}$ and thus the contribution of the region $\{|t|^\epsilon < |w| < \frac{1}{(\log|t|^{-1})^m}\}$ to the integral is  
$$\int_{|t|^\epsilon < |w| < \frac{1}{(\log|t|^{-1})^m}} f\left(r\left(w,\frac{\log|w|}{\log|t|}\right)\right) \frac{|dw \wedge d\overline{w}|}{2\pi l_{e_1} |w|^2 \log|t|^{-1}}.$$

Using the change of variables $u = \frac{\log|w|}{\log|t|^{-1}}$ and $\vartheta = \arg(w)$, we get
$$ \frac{1}{2 l_{e_1}\pi}\int_{\frac{m\log(\log|t|^{-1})}{\log|t|^{-1}}}^{\epsilon}\int_0^{2\pi} f(r(|t|^{u}e^{i\vartheta},u)) d\vartheta du.$$

As $t \to 0$, $f(r(|t|^{u}e^{i\vartheta},u)) \to f(r(0,u)) = f(u)$ pointwise almost everywhere and thus we get
$$ \int_{|t|^\epsilon < |w| < \frac{1}{(\log|t|^{-1})^m}} f\left(r\left(w,\frac{\log|w|}{\log|t|}\right)\right) \tau_t \to \frac{1}{l_{e_1}}  \int_{0}^\epsilon f(u)du.$$


In the region $\{\frac{1}{(\log|t|^{-1})^{m}} <  |w| < 1 \}$, note that $\tau_t \to \tilde{\tau}_0$.  Thus, it is enough to evaluate the limit
$$ \int_{\frac{1}{(\log|t|^{-1})^{m}}< |w| < 1} f\left(r\left(w,\frac{\log|w|}{\log|t|}\right)\right) \tilde{\tau}_0.$$
As $t \to 0$, $f\left(r\left(w,\frac{\log|w|}{\log|t|^{-1}}\right)\right) \to f(r(w,0)) = f(w)$. Thus, we get
$$ \int_{\frac{1}{(\log|t|^{-1})^{m}}< |w| < 1} f\left(r\left(w,\frac{\log|w|}{\log|t|}\right)\right) \tau_t  \to \int_{E_1 \cap U} f \tilde{\tau}_0$$
as $t \to 0$.
\end{proof}

\section{Convergence of $\mu_t$}
\label{secConvergenceCanonical}
Let the notation be as in the previous section. We prove the following precise version of Theorem \ref{mainThmB}.
\begin{thm}
\label{thmB}
Let $\X$ be the minimal snc model of $(X,\frac{1}{m}B)$.
Let $\mu_0$ denote the measure on $\Delta_\CC(\X)$ which is given by
\begin{itemize}
\item On a Type I component $E_i$, the pluri-Bergman measure on $E_i$ with respect to $mK_{E_i} + (m-1)P^{(i)}_1 + \dots + (m-1)P^{(i)}_{r_i} + \B|_{E_i}$.
\item If $E_i$ is a Type II component, we pick the zero measure
\item On each edge $e \in E(\Gamma_\X)$, we pick the Lebesgue measure of length $\frac{1}{l_e}$, where $l_e$ is the length of the edge in $\tilde{\Gamma}$ containing $e$.
\end{itemize}  
On $\X^\hyb_\CC$, the measures $\mu_t \to \mu_0$ weakly. 
\end{thm}
\begin{proof}
Theorem \ref{thmB} follows from Lemma \ref{lemConvergenceMCCHC} and Corollaries \ref{corB1} -- \ref{corB3}. 
\end{proof}

\subsection{Asymptotics of $\langle \theta_{j,t} , \overline{\theta_{k,t}} \rangle$}
\label{subsecAsymptotics2}

Recall from Section \ref{eqnGenBergmanPairing} the Hermitian pairing used to define the pluri-Bergman measure.
$$ \langle \theta, \vartheta \rangle = \left(\frac{i}{2}\right)^m \int \frac{\theta \wedge \overline{\vartheta}}{\tau^{m-1}}.$$
We first understand the asymptotics of this pairing.

Let $A(t)$ denote the $M \times M$ matrix with the $(j,k)$-th coefficient
$$ (A(t))_{j,k} = \langle \theta_{j,t}, \theta_{k,t} \rangle =  \left(\frac{i}{2}\right)^m \int_{X_t} \frac{\theta_{j,t} \wedge \overline{\theta_{k,t}}}{\tau_t^{m-1}} ,$$

Then using elementary linear algebra, we see that 
$$ \mu_t = \left(\frac{i}{2}\right)^m \sum_{j,k = 1}^M (\overline{A(t)})_{j,k}^{-1} \frac{\theta_{j,t} \wedge \overline{\theta_{k,t}}}{\tau_t^{m-1}} .$$

We first state a lemma that we will use to estimate the integral $\int_{X_t}\frac{\theta_{j,t} \wedge \overline{\theta_{k,t}}}{\tau_t^{m-1}}$.  
\begin{lem}
 \label{lemTrivialBoundIntegrand}
 There exists a constant $C$ such that 
 $$ \left| \frac{\theta_{j,t} \wedge \overline{\theta_{k,t}}}{\tau_t^{m-1}} \right| \leq C (\log|t|^{-1})^{(\eta_{j}+\eta_k)(m-1)} |\theta_{j,t} \wedge \overline{\theta_{k,t}}|^{1/m} $$
 for all $1 \leq j,k \leq M$, where $\eta_{j} = 0$ if $1 \leq j,k \leq s$ and  $\eta_{j,k} = 1$ if $s+1 \leq j,k \leq M$ and $\eta_{j,k} = \frac{1}{2}$ in all other cases.  
\end{lem}
\begin{proof}
From Corollary \ref{corTauTrivialBound}, we have that there exists a constant $C_1 > 0$ such that $\tau_t \geq C_1 \frac{|\theta_{j,t} \wedge \overline{\theta_{k,t}}|^{1/m}}{\log|t|^{\eta_{j} + \eta_k}}$. The result now follows immediately. 
\end{proof}

\begin{lem}
\label{lemAsymptoticsA}
  The matrix $A(t)$ has the following form. 
  $$A =
\begin{pmatrix}
  B & D^* \\
  D & F
\end{pmatrix}$$
  
\noindent where $B$  is an $s \times s$ matrix with entries $B_{j,j} \approx (2\pi l_{e_i} \log|t|^{-1})^{m}$ and $B_{j,k} = O((\log|t|^{-1})^{m-1})$ for $j \neq k$, $D = O((\log|t|^{-1})^{\frac{m-1}{2}})$ and $F$ is a $(M-s) \times (M-s)$ matrix such that $F \to \mathbf{I}_{M-s}$ as $t \to 0$.  
\end{lem}
\begin{proof}
  \newcommand{\integrand}{\frac{\theta_{j,t} \wedge \overline{\theta_{k,t}}}{\tau^{n-1}}}
  \newcommand{\errorA}{O\left(\frac{1}{(\log|t|^{-1})^{1/2}}\right)}
  We use a partition of unity argument to reduce the problem of computing the integral $\int_{X_t} \frac{\theta_{j,t} \wedge \overline{\theta_{k,t}}}{\tau_t^{m-1}}$ to computing it on adapted coordinate charts.
  The result follows by computing the integrals using Equations \eqref{eqnTheta1TR1}--\eqref{eqnThetaJTR2}, Corollary \ref{corAsymptoticsTau} and Lemma \ref{lemTrivialBoundIntegrand}. Since the leading term of $\tau$ is differs in different region of coordinate charts adapted to a node, we will have break up such a chart into different regions while analyzing the integral. We only show how to get the entries of $B$ and $F$. The estimates for entries of $C$ are obtained using similar techniques.

  To get the asymptotics for $F$, note that on a coordinate chart $(U_1,t,w)$ adapted to a Type I irreducible component $E$, we have that $\tau_t \to \tilde{\tau}_0$ and that $\tilde{\tau}_0|_{E}$ is nowhere vanishing. 
  
  Furthermore, there exists an integrable function on $\D$ that dominates $\left|\frac{\theta_{j,t} \wedge \overline{\theta_{k,t}}}{\tau_t^{m-1}}\right|$ for all $t$ small enough. One way to see this is by using Lemma \ref{lemTrivialBoundIntegrand}; we get  $\left|\frac{\theta_{j,t}(w) \wedge \overline{\theta_{k,t}(w)}}{\tau_t^{m-1}(w)}\right| \leq C |\theta_{j,t}(w) \wedge \overline{\theta_{k,t}(w)}|^{1/m} \leq C' |w|^{-2+1/m}$.   
  Thus by the dominated convergence theorem, we have that
  $$ \int_{U_1 \cap X_t} \frac{\theta_{j,t} \wedge \overline{\theta_{k,t}}}{\tau_t^{m-1}} \to \int_{U_1 \cap E} \frac{\psi_{j} \wedge \overline{\psi_{k}}}{\tilde{\tau}_0^{m-1}}$$
  as $t \to 0$ for all $s+1 \leq j,k \leq M$.

  If $E$ is of Type II, then using Lemma \ref{lemTrivialBoundIntegrand}, we have that $\left|\frac{\theta_{j,t} \wedge \overline{\theta_{k,t}}}{\tau_t^{m-1}}\right| \leq C |\theta_{j,t} \wedge \overline{\theta_{k,t}}|^{1/m}$. But since $\theta_{j,t} \to \psi_j = 0$ on $U_1$ for all $s+1 \leq j \leq M$, we get that
  $$ \left|\frac{\theta_{j,t} \wedge \overline{\theta_{k,t}}}{\tau_t^{m-1}}\right| \to 0 $$
  and
  $$ \int_{U_1 \cap X_t} \frac{\theta_{j,t} \wedge \overline{\theta_{k,t}}}{\tau_t^{m-1}} \to 0 $$
for all $s+1 \leq j \leq M$  if $E$ is of Type II. 

For a coordinate chart $(U_2,z,w)$ adapted to a node $P = E_1 \cap E_2$, to estimate the integral $\int_{|t|^{1/2} < |w| < 1} \frac{\theta_{j,t} \wedge \overline{\theta_{k,t}}}{\tau_t^{n-1}}$, observe that its pointwise limit is $\frac{\psi_{j} \wedge \overline{\psi_{k}}}{\tilde{\tau}_0^{m-1}}$ if $E_1$ is of Type I and is 0 if $E_1$ is of Type II. Since $\psi_{s+1}, \dots, \psi_M$ was chosen to be an orthonormal basis for the Hermitian pairing \eqref{eqnGenBergmanPairing}, we get the asymptotics for $F$.

To analyze the diagonal entries of $B$, without loss of generality consider $B_{11}$. The estimate using Lemma \ref{lemTrivialBoundIntegrand} and Equations \eqref{eqnThetaJTR1} and \eqref{eqnThetaJTR2} shows that $  \int_{U \cap X_t} \frac{\theta_{1,t} \wedge \overline{\theta_{1,t}}}{\tau_t^{m-1}} = O((\log|t|^{-1})^{m-1})$
if $U$ is either a coordinate chart adapted to an irreducible component, or a coordinate chart adapted to a node at which $\psi$ does not develop a pole of order $m$.

  To get the leading term for $B_{11}$, consider a coordinate chart $(U_2,z,w)$ adapted to a node $P$ such that $\psi_1$ has a pole of order $m$ at $P$. Consider the region $\{ |t|^{1/2} < |w| < \frac{1}{(\log|t|^{-1})^m} \}$.  Using Equation \eqref{eqnTheta1TR1} and Corollary \ref{corAsymptoticsTau}, we get that 
  $$
   \frac{|\theta_{1,t} \wedge \overline{\theta_{1,t}}|}{\tau_t^{m-1}} \approx \frac{(2\pi l_{e_1}\log|t|^{-1})^{m-1} |dw \wedge d \overline w|}{|w|^2}
  $$
 in this region.
 Then,
 \begin{align*}
 \int_{|t|^{1/2} < |w| < \frac{1}{(\log|t|^{-1})^m} }\frac{|\theta_{1,t} \wedge \overline{\theta_{1,t}}|}{\tau_t^{m-1}}
   &\approx \int_{|t|^{1/2} < |w| < \frac{1}{(\log|t|^{-1})^m}} \frac{(2\pi l_{e_1}\log|t|^{-1})^{m-1} |dw \wedge d \overline w|}{|w|^2} \\
   &=  2\pi \left( (2\pi l_{e_1}\log|t|^{-1})^{(m-1)} \int_{|t|^{1/2}}^{(\log|t|^{-1})^{-m}} \frac{dr}{r} \right)  \\
   &\approx \pi (2\pi l_{e_1})^{m-1}(\log|t|^{-1})^m.
 \end{align*}

 Using Lemma \ref{lemTrivialBoundIntegrand}, we get that  $$\int_{\frac{1}{(\log|t|^{-1})^m}< |w| < 1}\frac{|\theta_{1,t} \wedge \overline{\theta_{1,t}}|}{\tau_t^{m-1}}  \leq C (\log|t|^{-1})^{m-1}\int_{\frac{1}{(\log|t|^{-1})^m}< |w| < 1}|\theta_{1,t}|^{2/m} .$$ An easy computation shows that the integral on the right-hand side is of the order of $(\log|t|^{-1})^{m-1}\log(\log|t|^{-1})$, which is a subdominant term. Thus, we see that
 $$ \int_{|t|^{1/2}< |w| < 1} \frac{|\theta_{1,t} \wedge \overline{\theta_{1,t}}|}{\tau_t^{m-1}} \approx \pi (2\pi l_{e_1})^{m-1}(\log|t|^{-1})^m $$ and
 $$  \int_{U_2 \cap X_t} \frac{|\theta_{1,t} \wedge \overline{\theta_{1,t}}|}{\tau_t^{m-1}} \approx 2\pi (2\pi l_{e_1})^{m-1}(\log|t|^{-1})^m.$$

Since  $\psi_1$ has a pole of order $m$ along $l_{e_1}$ many nodes, we get that
 $$ B_{11} = \int_{X_t } \frac{|\theta_{1,t} \wedge \overline{\theta_{1,t}}|}{\tau_t^{m-1}} \approx (2\pi l_{e_1} \log|t|^{-1})^m.$$

 To estimate $B_{j,k}$ for $j \neq k$, using Lemma \ref{lemTrivialBoundIntegrand}, we only need to estimate $(\log|t|^{-1})^{m-1}|\theta_{j,t} \wedge \overline{\theta_{k,t}}|^{1/m}$. It is easy to verify using Equations \eqref{eqnTheta1TR1}--\eqref{eqnThetaJTR2} that $|\theta_{j,t} \wedge \overline{\theta_{k,t}}|^{1/m} = O(1)$ if $1 \leq j, k \leq s$ and $j \neq k$, which gives us the estimate for $B$. 

\end{proof}

The asymptotics of $A(t)^{-1}$ now follows by using Gauss-Jordan elimination. 
\begin{cor}
\label{corAsymptoticsPluriAInverse}
  The matrix $A(t)^{-1}$ has the following form. 
  $$A^{-1} =
\begin{pmatrix}
  B' & (D')^* \\
  D' & F'
\end{pmatrix}$$
  
\noindent where $B'$  is an $s \times s$ diagonal matrix with diagonal entries $$B'_{i,i} \approx \frac{1}{(2\pi l_{e_i} \log|t|^{-1})^{m}},$$
$B_{j,k} = O\left(\frac{1}{(\log|t|^{-1})^{m+1}}\right)$ for $i \neq j$,  $D' = O\left(\frac{1}{(\log|t|^{-1})^{\frac{m+1}{2}}}\right)$ and $F'$ is a $(M-s) \times (M-s)$ matrix such that $F' \to \mathbf{I}_{M-s}$ as $t \to 0$. \qed
\end{cor}

The following corollaries prove Theorem \ref{thmB}. These are an easy consequence of Equations \eqref{eqnTheta1TR1}--\eqref{eqnThetaJTR2} and Corollaries \ref{corAsymptoticsTau} and \ref{corAsymptoticsPluriAInverse}. 
\begin{cor}
 \label{corB1}
  Let $(U_1,t,w)$ be a coordinated chart adapted to an irreducible component $E \subset \X_0$. Let $f$ be a compactly supported function on $U_1$. If $E$ is of Type I, then
  $$\sum_{j,k=1}^M (\overline{A(t)})^{-1}_{j,k}\int_{U_1 \cap X_t} f \frac{\theta_{j,t} \wedge \overline{\theta_{k,t}}}{\tau^{m-1}_t} \to \sum_{j=s+1}^M \int_{U_1 \cap E} f\frac{\psi_j \wedge \overline{\psi}_j}{\tilde{\tau}^{m-1}_0}$$
  and if $E$ is Type II, then
  $$\sum_{j,k=1}^M (\overline{A(t)})^{-1}_{j,k}\int_{U_1 \cap X_t} f \frac{\theta_{j,t} \wedge \overline{\theta_{k,t}}}{\tau^{m-1}_t} \to 0 $$
  as $t \to 0$. \qed
\end{cor}

\begin{cor}
\label{corB2}
  Let $(U_2,z,w)$ be a coordinate chart adapted to a node $P = E_1 \cap E_2$. Let $f$ be a continuous function on $[0,1]$ and let $0 < \alpha < \beta < \frac{1}{2}$. Without loss of generality, suppose that $\psi_1$ develops a pole of order $m$ at $P$. Then,
  \begin{multline*}
 \sum_{j,k} (\overline{A(t)})^{-1}_{j,k}\int_{|t|^{\beta} < |w| < |t|^{\alpha}} f\left( \frac{\log|w|}{\log|t|}\right) \frac{ (\frac{i}{2})^m \theta_{j,t}(w) \wedge \overline{\theta_{k,t}(w)}}{\tau^{m-1}_t(w)} \to 
   \frac{1}{l_{e_1}} \int_{\alpha}^{\beta}f(u)du.
 \end{multline*} \qed
\end{cor}

\begin{cor}
 \label{corB3}
  Let $(U_2,z,w)$ be a coordinate chart adapted to a node $P = E_1 \cap E_2$. Let $0 < \epsilon \ll \frac{1}{2}$. Let $f$ be a continuous compactly-supported function on the half-dumbbell $D_\epsilon \subset \D \times [0,\epsilon)$ and  let $r : \D \times [0,\epsilon) \to D_\epsilon$ be a strong deformation retract. Without loss of generality, suppose that $\psi_1$ develops a pole of order $m$ at $P$. If $E$ is of Type I, then,
  \begin{multline*}
 \sum_{j,k} (\overline{A(t)})^{-1}_{j,k}\int_{|t|^{\epsilon} < |w| < 1} f\left(r\left(w,\frac{\log|w|}{\log|t|}\right)\right) \frac{(\frac{i}{2})^m\theta_{j,t}(w) \wedge \overline{\theta_{k,t}(w)}}{\tau^{m-1}_t(w)} \to \\
 \frac{1}{l_{e_1}} \int_{0}^{\epsilon}f(u)du +
 \sum^M_{j=s+1} \int_{U_2 \cap E_1}f \frac{|\psi_j \wedge \overline{\psi}_j|}{\tilde{\tau}^{m-1}_0}.
  \end{multline*}
  If $E$ is of Type II, then,
  \begin{multline*}
 \sum_{j,k} (\overline{A(t)})^{-1}_{j,k}\int_{|t|^{\epsilon} < |w| < 1} f\left(r\left(w,\frac{\log|w|}{\log|t|}\right)\right) \frac{(\frac{i}{2})^m\theta_{j,t}(w) \wedge \overline{\theta_{k,t}(w)}}{\tau^{m-1}_t(w)} \to 
 \frac{1}{l_{e_1}} \int_{0}^{\epsilon}f(u)du 
  \end{multline*}
\end{cor}
\begin{proof}
  We break up the region $\{ |t|^{\epsilon} < |w| < 1 \}$ into two regions: $\{|t|^{\epsilon} < |w| < \frac{1}{(\log|t|^{-1})^m} \}$ and $\{\frac{1}{(\log|t|^{-1})^m} < |w| < 1 \}$. We analyze the integral separately on each region.

  Let us first analyze the integral on the region $\{\frac{1}{(\log|t|^{-1})^{m}} < |w| < 1\}$. It follows from the asymptotics of $A^{-1}_{j,k}$, $\theta_{j,t}$ and $\tau_t$ that the pointwise limit of 
  $$\sum_{j,k} (\overline{A(t)})^{-1}_{j,k} f\left(r\left(w,\frac{\log|w|}{\log|t|}\right)\right) \frac{\theta_{j,t}(w) \wedge \overline{\theta_{k,t}(w)}}{\tau^{m-1}_t(w)} $$
  is $$ \sum^M_{j=s+1}f(w) \frac{\psi_j \wedge \overline{\psi}_j}{\tilde{\tau}^{m-1}_0} $$
  as $t \to 0$ if $E$ is of Type I, otherwise the limit is 0. Using the dominated convergence theorem, we get that
  \begin{multline*} \sum_{j,k} (\overline{A(t)})^{-1}_{j,k} \int_{\frac{1}{(\log|t|^{-1})^{1/4}} < |w| < 1} f\left(r\left(w,\frac{\log|w|}{\log|t|}\right)\right) \frac{\theta_{j,t}(w) \wedge \overline{\theta_{k,t}(w)}}{\tau^{m-1}_t(w)}  \to \\ \sum^M_{j=s+1} \int_{U_2 \cap E_1} f(w) \frac{\psi_j \wedge \overline{\psi}_j}{\tilde{\tau}^{m-1}_0}
  \end{multline*}
if $E$ is of Type I, otherwise the limit is 0. 

To estimate the integral in the region $\{|t|^{\epsilon} < |w| < \frac{1}{(\log|t|^{-1})^{m}}\}$, note that $\tau_{t} \approx \frac{|w|^{-2}}{2\pi l_{e_1}\log|t|^{-1}}$, $\theta_{1,t} \approx w^{-m}$ and $\theta_{j,t} = O(|w|^{1-m})$ for $2 \leq j \leq M$ uniformly in $t$. Thus,
$$ 
(\overline{A(t)})^{-1}_{j,k} \int_{|t|^{\epsilon} < |w| < \frac{1}{(\log|t|^{-1})^{m}} } \frac{\theta_{j,t}(w) \wedge \overline{\theta_{k,t}(w)}}{\tau^{m-1}_t(w)}  \to 0$$
unless $j=k=1$ and
\begin{align*}
  (\overline{A(t)})^{-1}_{1,1} \int_{|t|^{\epsilon} < |w| < \frac{1}{(\log|t|^{-1})^{m}}} f\left(r\left(w,\frac{\log|w|}{\log|t|}\right)\right) \frac{|\theta_{1,t}(w) \wedge \overline{\theta_{1,t}(w)}|}{\tau^{m-1}_t(w)}  \\
  \approx \frac{1}{2\pi l_{e_1} \log|t|^{-1}} \int_{|t|^\epsilon < |w| < \frac{1}{(\log|t|^{-1})^{m}}} f\left(r\left(w,\frac{\log|w|}{\log|t|}\right)\right) \frac{|dw \wedge d\overline{w}|}{|w|^2}
 \end{align*}

 Now we use a change of variables $u = \frac{\log|w|}{\log|t|}$ and $\vartheta = arg(w)$ to get
 $$ \frac{1}{2\pi l_{e_1}} \int_{\frac{m\log(\log|t|^{-1})}{\log|t|^{-1}}}^{\epsilon} \int_0^{2\pi} f(r(|t|^{u}e^{i\vartheta},u)) d\vartheta du .$$
 As $t \to 0$, we have $f(r(|t|^{u}e^{i\vartheta},u)) \to f(r(0,u)) = f(u)$. Thus, we get
\begin{multline*}
(\overline{A(t)})^{-1}_{1,1} \int_{|t|^{1/2} < |w| < \frac{1}{(\log|t|^{-1})^{m}}} f\left(r\left(w,\frac{\log|w|}{\log|t|}\right)\right) \frac{|\theta_{1,t}(w) \wedge \overline{\theta_{1,t}(w)}|}{\tau^{n-1}_t(w)} \to \\ \frac{1}{l_{e_1}} \int_0^\epsilon f(u)du
\end{multline*}  
\end{proof}

\subsection{Limit of $\mu_0$ as $m \to \infty$ for a fixed $B$}
\label{subsecLimitMToInfty}
We would like to understand the limit of $\mu_0$ as $m \to \infty$. Firstly, note that if $m$ is large enough then the minimal snc model of $(X,\frac{1}{m}B)$ is just the minimal semistable model of $X$. So, let $\X$ denote the minimal semistable model of $X$ in this section.

We only compute the limit on $\Gamma_\X$ and not on $\Delta_\CC(\X)$ as it is not clear to us what the limit behavior would look like near a smooth point in $\X_0$.

Let $\mu_t^{(m,B)}$ denote the pluri-Bergman measure induced by $\omega_{X_t}^{\otimes m}(B|_{X_t})$. Let $\mu_0^{(m,B)}$ be the weak limit $\mu_t^{(m,B)}$ on $X^\hyb$ as $t \to 0$. Then $\mu_0^{(m,B)}$ is a sum of Dirac mass on the vertices of $\Gamma_\X$ and Lebesgue measure on the edges on the edges of $\Gamma_\X$ for $m$ large enough. Note that the total mass of $\mu_0^{(m,B)}$ is:
$$ \mu_0^{(m,B)}(\Gamma_\X) =  \mu_t^{(m,B)}(X_t) = h^0(X_t,\omega_{X_t}(B|_{X_t})) = (2m-1)(g-1) + \deg(B|_{X_t}).$$

Let $$\mu_{0}^{(\infty)} := \lim_{m \to \infty} \frac{2(g-1)\mu_{0}^{(m,B)}}{(2m-1)(g-1)+\deg(B|_{X_t})}$$
be the weak limit of $\mu_0^{(m,B)}$ normalized to have volume $2g-2$. Here the limit is taken in the sense of Radon measures on $\Gamma_\X$. Let us compute $\mu_0^{(\infty)}$. 

Note that $\mu_0^{(m,B)}$ restricted to an edge $e$ is the Lebesgue measure of total mass $\frac{1}{N}$, where $N$ is the length of the maximal inessential chain
containing $e$. Note that $N$ is independent of $m$ for $m \gg 0$. Then,
$$\mu_0^{(\infty)}|_e = \lim_{m \to \infty} \frac{(2g-2)dx}{N((2m-1)(g-1)+\deg(B|_{X_t}))} = 0.$$
Thus, $\mu_0^{(\infty)}$ places no mass on the edges of $\Gamma_\X$.

Let $v$ be a vertex in $\Gamma_\X$ and let $E \subset \X_0$ be the associated irreducible component. It follows from Theorem \ref{thmB} that the mass of $\mu_0^{(m)}$ on $v_E$ is
$$ \mu_0^{(m,B)}(\{v_E\}) = h^0(mK_{E} + (m-1)\sum_{E' \neq E}(E \cap E') + \B|_{E})$$
if $E$ is a Type I component, otherwise it is 0. Note that for a fixed $B$ and $m$ large enough, being Type I is equivalent to being essential.

Thus,
$$ \mu_0^{(m,B)}(\{v_E\}) = (2g(E)-2)m + (m-1)\val(E) + \deg(\B|_{E})$$
if $E$ is essential, otherwise $\mu_0^{(m,B)}(v) = 0$.

We get that

\begin{align*}
  \mu_0^{(\infty)}(\{v_E\}) &= \lim_{m \to \infty} \frac{(2g-2)((2g(E)-2)m + (m-1)\val(E) + \deg(\B|_{E}))}{(2m-1)(g-1)+\deg(B|_{X_t})} \\
  &= 2g(E) - 2 + \val(E)
\end{align*}

if $E$ is essential, otherwise $\mu_0^{\infty}(v) = 0$. Thus, we get that
$$ \mu_0^{(\infty)} = \sum_{v \in V(\Gamma_\X)} (2g(v) - 2 + \val(v))\delta_v,$$
which is also the limit of the hyperbolic measures on $X_t$ \cite[Theorem A]{Sch19}.

\subsection{Limit of $\mu_0$ as $m \to \infty$ for fixed $\frac{1}{m}B$}
Another way to think of the limit of $\mu_0$ is to vary $B$, while fixing the $\Q$-divisor $\frac{1}{m}B$. Let $\mu_t^{(km,kB)}$ denote the pluri-Bergman measure on $X_t$ associated to $\Omega_{X_t}^{\otimes km}(kB|_{X_t})$. Let $\X$ denote the minimal snc model of $(X,\frac{1}{m}B)$. Note that $\X$ is also the minimal snc model of $(X,\frac{1}{km}kB)$ for all $k \geq 1$. Let $\mu_0^{(km,kB)}$ be the weak limit $\mu_t^{(km,kB)}$ as $t \to 0$ on $X^\hyb$. Then, $\mu_0^{(km,kB)}$ is supported on $\Gamma_\X$.  

Let $$\mu_0^{[\infty]} := \lim_{k \to \infty} \frac{(2g-2 + \frac{\deg(B|_{X_t})}{m})\mu_{0}^{(km,kB)}}{(2km-1)(g-1)+k\deg(B|_{X_t})} $$

be the limit of $\mu_0^{(km,kB)}$ normalized to volume $2g-2+ \frac{\deg(B|_{X_t})}{m}$.

A similar computation shows that the $\mu_{0}^{[\infty]}$ places no mass on the edges in $\Gamma_\X$.

The mass of $\mu_{0}^{[\infty]}$ on a vertex $v_E$ associated to an irreducible component $E$ is
\begin{align*}
  \mu_0^{[\infty]}(\{v_E\}) &= \lim_{k \to \infty} \frac{(2g-2+\frac{\deg(B|_{X_t})}{m})((2g(E)-2)km + (km-1)\val(E) + k\deg(\B|_{E}))}{(2km-1)(g-1)+k\deg(B|_{X_t})} \\
  &= 2g(E) - 2 + \val(E) + \frac{\deg(\B|_{E})}{m}
\end{align*}
if $E$ is a component of Type I, otherwise the limit is 0. Note that for $k \gg 0$, the only Types I components that can show up are the inessential components and those components $E$ for which $g(E) = 0$, $\val(E) = 1$ and $\deg(\B|_E) = m$. Since $2g(E)-2+\val(E) + \frac{\deg(\B|_E)}{m}$ is $0$ in both these cases, we get that
$$ \mu_0^{[\infty]} = \sum_{v_E \in V(\Gamma_\X)} \left(2g(v_E) - 2 + \val(v_E) + \frac{\deg(\B|_{E})}{m}\right)\delta_{v_E}.$$

\section{Convergence on $\Mgbar$}
\newcommand{\0}{\mathbf{0}}
\renewcommand{\t}{\mathbf{t}}
In this section, we show that in the case of $g \geq 2$ and $B=0$, $\tau_t$ extends to a continuous family of measures on the Deligne-Mumford compactification, $\Mgbar$, of the moduli space of genus $g$ curves. Let $g \geq 2$ and $m \geq 2$ be fixed integers. 

Let $\Cgbar$ denote the universal curve over $\Mgbar$.

Let $S_\0$ be a stable curve of genus $g$. Recall that this means that $S_\0$ is a reduced curve of arithmetic genus $g$ with only nodal singularities and every rational irreducible component of $S_\0$ intersects the rest of the curve in at least three points. We define the Narasimhan-Simha measure $\tau$ on $S_\0$ associated to $\omega_{S_\0}^{\otimes m}$ to be a Radon measure on $S_\0$ which will be a sum of Narasimhan-Simha measures on the irreducible components and Dirac masses at the nodal points. More precisely,
\begin{itemize}
\item Let $E_i \subset S_\0$ is an irreducible component and let $P_1^{(i)},\dots,P_{r_i}^{(i)}$ be the nodal points that lie on $E_i$. Then, $\tau|_{E_i \setminus \{P_1^{(i)},\dots,P^{(i)}_{r_i} \}}$ is the Narasimhan-Simha measure on $E_i$ associated to the line bundle $\O_{E_i}(mK_{E_i}+(m-1)P_1^{(i)}+ \dots + (m-1)P_{r_i}^{(i)}).$
\item If $P$ is a nodal point in $S_\0$, then $\tau$ has a unit Dirac mass at $P$ i.e.~$\tau(\{P\}) = 1$.   
\end{itemize}

\subsection{The local picture}
The first step in proving Theorem \ref{mainThmC} is to reduce it to a local computation. To do this, we need to understand the local charts on $\Mgbar$ and $\Cgbar$. The key tool used here is a \emph{Kuranishi family} \cite[Chapter XI]{ACG11}. 

Let $S_\0$ be a stable curve. Roughly speaking, a Kuranishi family parametrizes the local deformations of $S_\0$. More precisely, a family of stable curves $\pi : S \to D$ is said to be a Kuranishi family for $\pi^{-1}(\0) = S_\0$  if for any other family of stable curves $\pi'\colon S' \to D'$ with $\phi_0 : (\pi')^{-1}(x_0) \simeq S_\0$, there exists a neighborhood of $U' \subset D'$ of $x_0$ and unique maps $\phi$, $\psi$ which make the diagram commute. 
$$
\begin{tikzcd}
  (\pi')^{-1}(x_0)  \ar[r,"\phi_0"] \ar[d,hook]&
  S_\0 \ar[d,hook]\\
  (\pi')^{-1}(U') \ar[r, "\phi"] \ar[d,"\pi'"] &
  S \ar[d,"\pi"] \\
  (U',x_0) \ar[r,"\psi"] &
  (D,\0)
\end{tikzcd}
$$

From this universal property, it follows that, up to shrinking the base, a Kuranishi family is unique up to an isomorphism. A Kuranishi family always exists for a stable curve. The total space $S$ is regular and the base $D$ is smooth over $\C$ and has dimension $3g-3$. We can choose coordinates $\t = (t_1,\dots,t_{3g-3})$ on $D$ so that $D \simeq \D^{3g-3}$ and the point $\0$ can be identified with the origin. Moreover, the coordinates $\t$ can be chosen in such a way a that a fiber over a point $\t$ has a node corresponding to $i$ iff $t_i = 0$, where $i=1,\dots,s$ is an enumeration of nodes in $S_\0$.  See \cite[Chapter XI]{ACG11} for details.

Thus, we can think of the coordinates $t_i$ as a coordinate that smoothens out the node $i$ for $i=1,\dots,s$ and the coordinates $t_i$ as varying the complex structure on the irreducible components of $S_\0$ for $i=s+1,\dots,3g-3$.

For $\t \in D$, let $S_\t$ denote the fiber over the point $\t$. By the adjunction formula, for a point $\t_0 \in D$, we have  $$\omega_S\left(\sum_{i=1}^{3g-3} \div(t_i - t_{0,i})\right) \Bigg|_{S_{\t_0}} \simeq \omega_{S_{\t_0}}.$$
Since $\div(t_i-t_{0,i})$ are principal divisors, we get that
$$ \omega_{S}|_{S_\t} \simeq \omega_{S_{\t}} $$ for all $\t \in D$, where the isomorphism is given by `unwedging' $dt_1\wedge\dots\wedge dt_{3g-3}$. 

We also get that $ \omega_{S}^{\otimes m}|_{S_\t} \simeq \omega^{\otimes m}_{S_{\t}} $ for all $\t \in D$. Since $h^0(S_\t,\omega_{S_\t}) = (2m-1)(g-1)$ is independent of $\t$, using Grauert's lemma \cite[Corollary III.12.9]{Har77}, we get that $\pi_*(\omega_{S}^{\otimes m})$ is a locally free sheaf on $B$ of rank $M = (2m-1)(g-1)$. 

Thus, we get an analog of Lemma \ref{lemExtendSections} i.e.~possibly after shrinking $D$, there exists $\theta_1,\dots,\theta_{M} \in H^0(S,\omega_S)$ such that $\theta_{i,\t} = \theta_i|_{S_\t}$ for $i=1,\dots,M$ form a basis of $H^0(S_\t,\omega_{S_\t})$ for all $\t \in D$. Recall that by $\theta_{i}|_{S_\t}$, we mean that we `unwedge' $(dt_1\wedge\dots\wedge dt_{3g-3})^{\otimes m}$ from $\theta_i$ and then restrict it to $S_\t$.

After performing a change of basis, we can assume that $\theta_i|_{S_\0}$ only has a pole of order $m$ with residue 1 along the $i$-th node in $S_\0$ for all $i=1,\dots,s$ and $\theta_i$ does not have a pole of order $m$ for all $i = s+1,\dots,M$. Now all we need to do is to repeat the analysis in Section \ref{secConvergenceNS}.

The $i$-th nodal point $P_i \in S_\0$ has a neighborhood $U$ in $S$ with coordinates $$(t_1,\dots,t_{i-1},z_i,w_i,t_{i+1},\dots,t_{3g-3})$$ such that the projection to $D$ is given by $t_i = z_iw_i$. Similar to the computation in \ref{secConvergenceNS}, we get that
$$ \theta_{j,\t} = \sum_{\alpha,\beta\geq 0,{\gamma} \in \N_{\geq 0}^{3g-4}} c^{(j)}_{\alpha,\beta,\gamma} t_i^{\alpha} w_i^{\beta - \alpha - m}(\mathbf{t'}_i)^{\gamma}dw_i^{\otimes m},$$
where $\mathbf{t'}_i = (t_1,\dots,\widehat{t_i},\dots,t_{3g-3})$ and $ |t_i|^{1/2} < |w_i| < 1 $.
Here, the residue of $\theta_{j,\0}$ along $P_i$ up to a sign is $c^{(j)}_{0,0,0}$.
Thus, $|c^{(j)}_{0,0,0}| = \delta_{ij}$ for  $j=1,\dots,M$. Similarly, we can also get such an expression in terms of the $z_i$ coordinate and we can also find an expression of $\theta_j$ in a neighborhood of smooth points of $S_\0$. 

We get an analog of Lemma \ref{lemAsymptoticsPseudonorm} by following the same techniques.
\begin{lem} For $i=1,\dots,s$, we have
  $$\| \theta_{i,\t} \|'_{S_\t} = (2\pi\log|t_i|^{-1})^{m/2} + O(1)$$
  and for $i = s+1,\dots,M$, we have
  $$\| \theta_{i,\t} \|'_{S_\t} \to \| \theta_{i,\0} \|'_{\tilde{S_\0}}$$
  as $\t \to 0$. \qed
\end{lem}

Let $\tau_\t$ denote the Narasimhan-Simha measure on $S_\t$. We also get an analog of Corollary \ref{corAsymptoticsTau}
\begin{lem}
  In the chart $U$ around $P_i$ described above, in the region $\{ |t_i|^{1/2} < |w_i| < \frac{1}{(\log|t_i|^{-1})^{m}}\}$ we have that
  $$\tau_\t \approx \frac{|dw_i \wedge d\overline{w_i}|}{2\pi |w_i|\log|t_i|^{-1}}.$$

  Away from such a region and near smooth points of $S_\0$, we have that $\tau_{\t} \to \tilde{\tau}_{\0}$, where $\tilde{\tau}_\0$ is the part of $\tau_\0$ without the Dirac masses. \qed
\end{lem}

We immediately get the following result, which is a local version of Theorem \ref{mainThmC}.
\begin{cor}
\label{corConvergenceKuranishi}
Let $f$ be a continuous compactly supported function on $S$. Then, $\int_{S_\t} f \tau_\t \to \int_{S_\0} f \tau_\0$ as $\t \to 0$. 
\end{cor}
\begin{proof}
  By using partitions of unity, we can assume that support of $f$ is small enough. If $f$ is supported in the chart $U$ described above, then using the fact that $$\int_{|t_i|^{1/2} < |w_i| < \frac{1}{(\log|t_i|^{-1})^{m}}} \frac{|dw_i \wedge d\overline{w_i}|}{2\pi |w_i|\log|t_i|^{-1}} \to \frac{1}{2},$$ we get
  \begin{align*}
\int_{|t_i|^{1/2} < |w_i| < 1} f \tau_\t &= \int_{|t_i|^{1/2} < |w_i| < \frac{1}{(\log|t_i|^{-1})^{m}}} f \tau_\t + \int_{\frac{1}{(\log|t_i|^{-1})^{m}} < |w_i| < 1} f \tau_\t
  \\ &\to \frac{f(P_i)}{2} + \int_{0 < |w_i| < 1} f \tilde{\tau}_{\0} 
  \end{align*}
  and
  $$ \int_{U \cap S_\t} f \tau_\t \to f(P_i) + \int_{U \cap S_\0} f \tilde{\tau}_\0  = \int_{U \cap S_\0} f \tau_\t $$ 
  as $\t \to 0$.

We also get a similar convergence near the smooth points of $S_\0$, which proves the result.   
\end{proof}

\begin{rmk}[Limit of pluri-Bergman measures]
  If we try to figure out the limit of the pluri-Bergman measures on $S$ using the techniques in Section \ref{secConvergenceCanonical}, then we run into an issue. As an analog of Lemma \ref{lemAsymptoticsA}, we will get that the matrix $A$ will have the first $s$ diagonal entries being $(2\pi\log|t_i|^{-1})^m$, but when we try to figure out the asymptotics of $A^{-1}$, we see that it will depend on the relative orders of magnitude of $\log|t_i|^{-1}$. Thus, it is unlikely that the limit of $\mu_t$ will exist on $S$. The limit of $\mu_t$ might exist on the hybrid space constructed by Amini and Nicolussi \cite{AN20} which keeps track of the order of magnitude of $\log|t_i|^{-1}$.
\end{rmk}

\subsection{The global picture}
Let $\pi : S \to D$ denote a Kuranishi family for a stable curve $S_\0$. Let $G = \Aut(S_\0)$ denote the (finite) group of self biholomorphisms of $S_\0$. From the universal property of the Kuranishi family, we see that $G$ acts on $S$ as well as $D$, after possibly shrinking $D$. 

Since $G$ is finite, the quotients $S/G$ and $D/G$ exist as normal complex analytic spaces. For our purposes, the underlying topological space is sufficient. 
The spaces $D/G$ and $S/G$ forms a neighborhood of the isomorphism class of $S_\0$ in $\Mgbar$ and its preimage in $\Cgbar$. Locally, the map $\Cgbar \to \Mgbar$ is given by $S/G \to D/G$. 
$$
\begin{tikzcd}[column sep = 4pt]
  S/G \ar[d] &\subset &
  \Cgbar \ar[d]\\
  D/G &\subset &
  \Mgbar  
\end{tikzcd}
$$

Note that $S \to D$ is also a Kuranishi family for $S_\t$ for all $\t \in D$ \cite[Corollary XI.4.9]{ACG11}. Thus, it follows that the stabilizer of a point $\t \in D$ under the action $G$ is $\Aut(S_\t)$. Let $\t,\t' \in D$ be two points in the orbit of the $G$-action on $D$. Then, the action of some element in $G$ provides a biholomorphism $S_\t \xrightarrow{\sim} S_{\t'}$ which induces a canonical biholomorphic map $S_{\t}/\Aut(S_\t) \xrightarrow{\sim} S_{\t'}/\Aut(S_{\t'})$ and is independent of the choice of the aforementioned element of $G$.

Thus, topologically, the fiber of the map $\Cgbar \to \Mgbar$ over the isomorphism class of a stable curve $C$ is $C/\Aut(C)$.

Recall from the construction of $\tau$ on a smooth genus $g$ curve $Y$ (see Section \ref{subsecPluriBergman}) that $\tau$ is invariant under the action of $\Aut(Y)$. It is not hard to check that $\tau$ is also invariant on a stable curve under the action of its automorphism group.

Let $\tau'_\t$ denote the pushforward of the Narasimhan-Simha measure under the map $S_\t \to S_\t/\Aut(S_\t)$. 

The following corollary is equivalent to Theorem \ref{mainThmC}
\begin{cor}
  The map $D/G \to (\mathcal{C}_c(S/G))^\vee$ given by sending $[\t] \mapsto \tau'_{\t}$ is well defined and continuous, where $(\mathcal{C}_c(S/G))^\vee$ is the space of Radon measures on $S/G$ equipped with the weak$^*$ topology. 
\end{cor}
\begin{proof}
  From Lemma \ref{corConvergenceKuranishi}, it follows that the map $D \to (\mathcal{C}_c(S))^\vee$ given by $\t \mapsto \tau_\t$ is continuous. It is enough to show that the composition $D \to (\mathcal{C}_c(S))^\vee \to (\mathcal{C}_c(S/G))^\vee$ is invariant under the action of $G$ on $D$ i.e.~we need to show that if $g \cdot \t = \t'$ for $g \in G$ and $\t,\t' \in D$, then the pushforward of $\tau'_\t$ is the same as $\tau'_{\t'}$ under the canonical identification $S_{\t}/\Aut(S_\t) \xrightarrow{\sim} S_{\t'}/\Aut(S_{\t'})$.

  Consider the diagram
  $$
  \begin{tikzcd}
    S_\t \ar[r,"g"] \ar[d] &
    S_{\t'} \ar[d]\\
    S_{\t}/\Aut(S_\t) \ar[r,"\sim"] &
    S_{\t'}/\Aut(S_{\t'})
  \end{tikzcd}
  $$

 Since $g$ induces a biholomorphism between $S_\t$ and $S_{\t'}$, it follows that the pushforward of $\tau_\t$ under $g$ is the same as $\tau_{\t'}$ and thus we get that the pushforward of $\tau'_{t}$ to $S_{\t'}/\Aut(S_{\t'})$ is the same the same as $\tau'_{\t'}$
\end{proof}

\bibliography{biblio}
\bibliographystyle{halpha.bst}
\end{document}